\documentclass[11pt,twoside,a4paper]{article}

\usepackage[australian]{babel}

\usepackage[driver=pdftex,margin=3cm,heightrounded=true,centering,includeheadfoot]{geometry}

\usepackage{hyperref}

\usepackage{amssymb,amsmath,amsfonts}
\usepackage{amsthm}
\usepackage[all]{xy}
\usepackage[capitalise]{cleveref}

\allowdisplaybreaks
\numberwithin{equation}{section}

\usepackage{xcolor}
\definecolor{MyBlue}{cmyk}{1,0.13,0,0.63}
\definecolor{MyGreen}{cmyk}{0.91,0,0.88,0.52}
\newcommand{\mylinkcolor}{MyBlue}
\newcommand{\mycitecolor}{MyGreen}
\newcommand{\myurlcolor}{webbrown}

\makeatletter
\def\@endtheorem{\endtrivlist}
\makeatother
\theoremstyle{plain}
\newtheorem{thm}{Theorem}[section]
\newtheorem{main}{Theorem}

\newtheorem*{main*}{Main Theorem}
\newtheorem{lem}[thm]{Lemma}
\newtheorem{prop}[thm]{Proposition}
\newtheorem{coro}[thm]{Corollary}

\newtheorem*{question}{Question}
\theoremstyle{definition}
\newtheorem{defn}[thm]{Definition}
\newtheorem{remark}[thm]{Remark}

\newtheoremstyle{note}
{3pt}
{3pt}
{\bfseries}
{\parindent}
{\bfseries\itshape}
{:}
{.5em}
{}
\theoremstyle{note}
\newtheorem*{Note}{Note}

\renewcommand{\eqref}[1]{\labelcref{#1}}
\crefname{thm}{Theorem}{Theorems}
\crefname{lem}{Lemma}{Lemmas}
\crefname{prop}{Proposition}{Propositions}
\crefname{coro}{Corollary}{Corollaries}
\crefname{defn}{Definition}{Definitions}
\crefname{example}{Example}{Examples}
\crefname{remark}{Remark}{Remarks}

\hypersetup{%
  bookmarksnumbered=true,bookmarksopen=false,%
  plainpages=false,
  linktocpage=true,%
  colorlinks=true,breaklinks=true,%
  linkcolor=\mylinkcolor,citecolor=\mycitecolor,urlcolor=\myurlcolor,%
  pdfpagelayout=OneColumn,%
  pageanchor=true,%
}

\makeatletter
\def\thm@space@setup{%
  \thm@preskip=4pt plus 2pt minus 2pt
  \thm@postskip=\thm@preskip
}
\renewenvironment{proof}[1][\proofname]{\par
  \pushQED{\qed}%
  \normalfont \topsep4\p@\relax 
  \trivlist
  \item[\hskip\labelsep
        \itshape
    #1\@addpunct{.}]\ignorespaces
}{%
  \popQED\endtrivlist\@endpefalse
}
\makeatother

\usepackage[shortlabels]{enumitem}
\setlist{topsep=4pt plus 2pt minus 2pt,partopsep=0pt,itemsep=2pt plus 2pt minus 2pt,parsep=0.5\parskip}
\setenumerate[1]{label=(\arabic*)}

\usepackage{fancyhdr}
\newcommand{\MR}[1]{}

\let\OLDthebibliography\thebibliography
\renewcommand\thebibliography[1]{
  \addcontentsline{toc}{section}{\refname}
  \OLDthebibliography{GMR19}
  \setlength{\parskip}{0pt}
  \setlength{\itemsep}{0pt plus 0.3ex}
}

\newcommand{\N}{\mathbb{N}}
\newcommand{\R}{\mathbb{R}}
\newcommand{\C}{\mathbb{C}}
\newcommand{\Z}{\mathbb{Z}}

\newcommand{\A}{\mathcal{A}}

\newcommand{\D}{\mathcal{D}}
\newcommand{\E}{\mathcal{E}}
\newcommand{\mS}{\mathcal{S}}

\DeclareMathOperator{\Dom}{Dom}
\DeclareMathOperator{\Ran}{Ran}

\DeclareMathOperator{\Ker}{Ker}

\DeclareMathOperator{\Id}{Id}
\DeclareMathOperator{\id}{id}
\DeclareMathOperator{\End}{End}
\DeclareMathOperator{\Hom}{Hom}
\DeclareMathOperator{\ev}{ev}
\DeclareMathOperator{\Lip}{Lip}
\DeclareMathOperator{\sgn}{sgn}
\DeclareMathOperator{\spec}{spec}

\newcommand{\K}{K}
\newcommand{\KK}{K\!K}
\newcommand{\UKK}{U\!K\!K}
\newcommand{\op}{\textnormal{op}}
\newcommand{\til}[1]{\widetilde{#1}}
\newcommand{\hotimes}{\mathbin{\hat\otimes}}
\newcommand{\hot}{\hotimes}

\newcommand{\la}{\langle}
\newcommand{\ra}{\rangle}
\newcommand{\into}{\hookrightarrow}
\newcommand{\mvert}{\,|\,}

\newcommand{\Bigmvert}{\,\Big|\,}
\renewcommand{\bar}[1]{\overline{#1}}
\newcommand{\order}{\mathcal{O}}
\newcommand{\sgnmod}{\textnormal{sgnmod}}
\newcommand{\sgnlog}{\textnormal{sgnlog}}
\renewcommand{\hat}{\widehat}
\newcommand{\CCliff}{{\mathbb{C}\mathrm{l}}}

\newcommand{\mattwo}[4]{
  \left(\!\!\!\begin{array}{c@{~}c}#1&#2\\ #3&#4\\\end{array}\!\!\!\right)
}

\title{Homotopy equivalence in unbounded \texorpdfstring{$\KK$}{KK}-theory}

\author{
Koen van den Dungen$^{1}$ and Bram Mesland$^{2}$
\\[2mm]
{\small ${}^1$Mathematisches Institut}, 
{\small Universit\"at Bonn}\\
{\small Endenicher Allee 60, D-53115 Bonn, Germany}\\
{\small \texttt{kdungen@uni-bonn.de}}\\[2mm]
{\small ${}^2$IMAPP, Radboud University Nijmegen}\\
{\small Heyendaalseweg 135, 6525AH Nijmegen, the Netherlands}\\
{\small \texttt{b.mesland@math.ru.nl}}
}

\begin{document}

\maketitle
\begin{abstract}
\noindent
We propose a new notion of unbounded $\KK$-cycle, mildly generalising unbounded Kasparov modules, for which the direct sum is well-defined. 
To a pair $(A,B)$ of $\sigma$-unital $C^{*}$-algebras, we can then associate a semigroup $\overline{\UKK}(A,B)$ of homotopy equivalence classes of unbounded cycles, and we prove that this semigroup is in fact an abelian group. 
In case $A$ is separable, our group $\overline{\UKK}(A,B)$ is isomorphic to Kasparov's $\KK$-theory group $\KK(A,B)$ via the bounded transform. 
We also discuss various notions of degenerate cycles, and we prove that the homotopy relation on unbounded cycles coincides with the relation generated by operator-homotopies and addition of degenerate cycles.

\vspace{\baselineskip}
\end{abstract}

\phantomsection
\pdfbookmark[1]{Introduction}{Introduction}
\section*{Introduction}

Given two ($\sigma$-unital, $\Z_2$-graded) $C^*$-algebras $A$ and $B$, Kasparov \cite{Kas80} defined the abelian group $\KK(A,B)$ as a set of homotopy equivalence classes of Kasparov $A$-$B$-modules, equipped with the direct sum. 
These groups simultaneously generalise $\K$-theory (if $A=\C$) and $\K$-homology (if $B=\C$). 

It was shown by Baaj-Julg that every class in $\KK(A,B)$ can also be represented by an \emph{unbounded} Kasparov module. 
Many examples of elements in $\KK$-theory which arise from geometric situations are most naturally described in the unbounded picture. The prototypical example is a first-order elliptic differential operator (e.g.\ the Dirac operator, signature operator, or de Rham operator) on a complete Riemannian manifold. 
The unbounded picture is also more suitable in the context of non-smooth manifolds. Indeed, while on Lipschitz manifolds there is no pseudodifferential calculus, it makes perfect sense to consider first-order differential operators and thus to construct unbounded Kasparov modules on Lipschitz manifolds (see e.g.\ \cite{Tel83,Hil85,Hil89}). 
Furthermore, the Kasparov product is often easier to describe in the unbounded picture. In fact, under suitable assumptions, the Kasparov product of two unbounded Kasparov modules can be explicitly \emph{constructed} \cite{Mes14,KL13,BMS16, MR16}.
These advantages of the unbounded picture of $\KK$-theory motivate the following question:

\begin{question} 
Can Kasparov's $\KK$-groups equivalently be defined as the set of homotopy equivalence classes of \emph{unbounded} Kasparov modules?
\end{question}

A similar question is considered in \cite{Kaa19pre}, where it is shown that Kasparov's $\KK$-groups can be obtained using the (a priori) weaker equivalence relation of \emph{stable homotopy} of unbounded Kasparov modules. 
In the present paper we will provide a positive answer to the above Question. 
Moreover, we will prove that the stable homotopy relation of \cite{Kaa19pre} in fact coincides with ordinary homotopy equivalence.

The first problem one encounters when trying to answer the above Question, is that the direct sum of unbounded Kasparov modules is not well-defined. To resolve this issue, we slightly weaken the standard definition of unbounded Kasparov modules  in such a way that the set $\overline{\Psi}_1(A,B)$ of such \emph{unbounded $A$-$B$-cycles} $(E,\D)$ becomes closed under the direct sum operation. 
By considering the natural notion of homotopy equivalence on $\overline{\Psi}_1(A,B)$ (completely analogous to homotopies of bounded Kasparov modules), we thus obtain a semigroup $\overline{\UKK}(A,B)$ given by the set of homotopy equivalence classes of $\overline{\Psi}_1(A,B)$.
We will prove that $\overline{\UKK}(A,B)$ is in fact a group. 

To answer the aforementioned Question, we need to show that the group $\overline{\UKK}(A,B)$ is isomorphic to Kasparov's $\KK$-theory group $\KK(A,B)$. 
The results of Baaj-Julg already show that the \emph{bounded transform}
\[(E,\D)\mapsto (E, F_{\D}:=\D(1+\D^{2})^{-\frac{1}{2}}),\]
induces a surjective homomorphism $\overline{\UKK}(A,B) \to \KK(A,B)$. This is proven by explicitly constructing an unbounded lift for any bounded Kasparov module. 

The difficulty is to prove injectivity of the bounded transform.
To be precise, given unbounded cycles $(E_0,\D_0)$ and $(E_1,\D_1)$ and a homotopy $(E,F)$ between their bounded transforms, we can use the lifting results from Baaj-Julg to lift $(E,F)$ to an unbounded homotopy $(E,\mS)$. However, it is in general not clear how the endpoints of $(E,\mS)$ are related to $(E_j,\D_j)$, and the main challenge is therefore to construct $(E,\mS)$ in such a way that its endpoints are in fact homotopic to $(E_j,\D_j)$. 

For this purpose, we describe a general notion of \emph{functional dampening}, which is the transformation $\D \mapsto f(\D)$ for suitable `dampening functions' $f\colon\R\to\R$ which blow up towards infinity at a slow enough rate (such that $f(x)(1+x^2)^{-\frac12}$ vanishes at infinity) and which are compatible with the Lipschitz structure obtained from $\D$. 
We prove that $(E,f(\D))$ is operator-homotopic to $(E,\D)$ for any dampening function $f$, generalising a result in \cite{Kaa19pre}.

By a careful adaptation of the lifting construction of \cite{BJ83} and \cite{Kuc00}, using ideas from \cite{MR16}, we then prove our first main result: 
\begin{main}
\label{main:lift}
If $A$ is separable, then any homotopy $(E,F)$ between $(E_0,F_{\D_0})$ and $(E_1,F_{\D_1})$ can be lifted to an unbounded Kasparov $A$-$C([0,1],B)$-module $(E,\mS)$ such that, for $j=0,1$, the endpoints $\ev_j(E,\mS)$ are unitarily equivalent to $(E_j,f_j(\D_j))$ for dampening functions $f_j\colon\R\to\R$. 
\end{main}

As mentioned above, functional dampening provides an operator-homotopy between $(E_j,\D_j)$ and $(E_j,f_j(\D_j))$, and thus we obtain a \emph{positive answer} to the above Question:

\begin{main}
\label{main:isomorphism}
If $A$ is separable, then the bounded transform induces an isomorphism 
\[
\overline{\UKK}(A,B) \xrightarrow{\simeq} \KK(A,B) . 
\]
\end{main}

We continue to provide an alternative description of the homotopy equivalence relation at the unbounded level. In bounded $\KK$-theory, it is well known that the homotopy relation coincides with the relation obtained from unitary equivalences, operator-homotopies, and addition of degenerate modules. 
We will prove an analogous statement in unbounded $\KK$-theory. We consider two notions of degenerate cycles, namely \emph{spectrally degenerate} cycles (for which $\D$ is invertible and $\D|\D|^{-1}$ commutes with $A$) and \emph{algebraically degenerate} cycles (for which $A$ is represented trivially). 
We then consider the equivalence relation $\sim_{oh+d}$ obtained from unitary equivalences, operator-homotopies, and addition of algebraically and spectrally degenerate cycles. Our next main result then reads: 

\begin{main}
\label{main:op-hom_mod_null}
Degenerate cycles are null-homotopic. 
Furthermore, if $A$ is separable, then the homotopy equivalence relation $\sim_h$ on $\overline{\Psi}_1(A,B)$ coincides with the equivalence relation $\sim_{oh+d}$. 
\end{main}

We prove the first statement by explicitly constructing a homotopy between degenerate cycles and the zero cycle. 
The second statement is then obtained by combining \cite[\S6, Theorem 1]{Kas80} with \cref{main:lift}. 

Let us briefly compare our work with the existing literature on unbounded Kasparov modules. 
First, we note that, in the usual approach to unbounded $\KK$-theory, it is necessary to make a fixed choice of a dense $*$-subalgebra $\A\subset A$, and to consider only those unbounded Kasparov $A$-$B$-modules $(E,\D)$ for which $\A\subset\Lip(\D)$, to ensure that the direct sum is well-defined. 
This means that any equivalence relation on unbounded Kasparov $A$-$B$-modules only applies to those unbounded Kasparov modules which are defined using the \emph{same} choice of $\A$. 
Thus it is impossible to compare unbounded Kasparov modules which are defined with respect to \emph{different} choices of $\A$. 
One major advantage of our approach is that, instead of fixing a choice of $*$-subalgebra $\A$, we consider the slightly weaker notion of \emph{unbounded cycles}, which only requires that $A\subset\overline{\Lip(\D)}$. For such cycles the direct sum is well-defined in full generality. In particular, the notion of homotopy equivalence can then be used to compare \emph{arbitrary} unbounded $A$-$B$-cycles. 
Nevertheless, we will show that Theorems \ref{main:lift}-\ref{main:op-hom_mod_null} remain valid if we do fix a countably generated dense $*$-subalgebra $\A\subset A$, and replace $\overline{\UKK}(A,B)$ by the semigroup $\UKK(\A,B)$ given by homotopy equivalence classes of all those unbounded Kasparov modules $(\pi,E,\D)$ for which $\pi(\A) \subset \Lip(\D)$. 

Other equivalence relations on unbounded Kasparov modules have already been considered in the literature, namely the bordism relation \cite{DGM18} and the stable homotopy relation \cite{Kaa19pre}. Both of these approaches rely on a fixed choice of a dense $*$-subalgebra $\A\subset A$. 
Let us discuss the relationships between homotopy equivalence, stable homotopy equivalence, and bordism. 
The paper \cite{DGM18} studies a notion of bordism of unbounded Kasparov modules due to Hilsum \cite{Hil10}, and proves that there is a surjective homomorphism from the corresponding bordism group $\Omega(\A,B)$ to Kasparov's $\KK$-group $KK(A,B)$. In particular, from \cref{main:isomorphism} we obtain a surjective homomorphism to our $\overline{\UKK}$-group, which means that the bordism relation is weaker than the homotopy relation. However, it remains an open question if these relations coincide or not. 
One technical tool appearing in \cite{DGM18} is the notion of weakly degenerate module, which is shown to be null-bordant. As a spin-off from our study of Clifford symmetric modules, we give a direct proof in \cref{lem:weakly_deg} that any weakly degenerate cycle is also null-homotopic (without assuming $A$ to be separable). 

After the appearance of \cite{DGM18} as a preprint in 2015, there has been increased interest within the community regarding equivalence relations on unbounded Kasparov modules. Discussions between the authors and Kaad in November 2018 gave the problem new impetus. The subsequent paper \cite{Kaa19pre} by Kaad provides a first study of homotopies of unbounded Kasparov modules. The work in the present paper was initiated independently and the methods developed here are complementary to those in \cite{Kaa19pre}. The main technical results, our Theorem A and \cite[Proposition 6.2]{Kaa19pre} are very distinct in spirit and lend themselves to different types of applications. Our proofs of Theorems \ref{main:lift}-\ref{main:op-hom_mod_null} are independent of the results from \cite{Kaa19pre}. Moreover it should be noted that our Theorem B is stronger than the main result in \cite{Kaa19pre} in the sense we now explain.

In \cite{Kaa19pre}, Kaad fixes a countably generated dense $*$-subalgebra $\A\subset A$ and considers the notion of \emph{stable homotopy} of unbounded Kasparov $\A$-$B$-modules. Stable homotopy is a weakening of the homotopy equivalence relation obtained from homotopy equivalences and addition of `spectrally decomposable' modules. It is then proved that the resulting set of equivalence classes of unbounded Kasparov $\A$-$B$-modules forms an abelian group which (if $A$ is separable) is isomorphic to Kasparov's $\KK$-group. In particular, this group does not depend on the choice of the dense $*$-subalgebra $\A\subset A$ (up to isomorphism).

As described above, we avoid in the present paper the need to fix a countably generated dense $*$-subalgebra $\A\subset A$ in the definition of the unbounded $\KK$-group. 
Even more importantly, thanks to our new approach towards lifting a homotopy in \cref{main:lift} (adapting the more refined lifting methods of \cite{Kuc00, MR16}), we overcome the need to weaken the homotopy equivalence relation by addition of spectrally decomposable modules. 
Furthermore, we will also show that, in fact, adding spectrally decomposable modules does not weaken the homotopy equivalence relation after all. Indeed, any spectrally decomposable module is just a bounded perturbation of a spectrally degenerate module. Consequently, it follows from \cref{main:op-hom_mod_null} that any spectrally decomposable cycle is null-homotopic, so that the relation of stable homotopy equivalence coincides with homotopy equivalence. 
We point out that, combined with the main results from \cite{Kaa19pre}, this provides a second and independent proof of \cref{main:isomorphism}. 

Finally, let us briefly summarise the layout of this paper. We start in \cref{sec:Kasmod} with our definition of unbounded cycles, and we show that the direct sum is well-defined. In \cref{sec:C_J,sec:lift} we recall the lifting construction from \cite{BJ83}, closely following the arguments of \cite{MR16,Kuc00}. 
We collect some basic facts regarding regular self-adjoint operators in \cref{sec:reg_sa}. 

In \cref{sec:semigroup} we introduce the homotopy relation (as well as the special case of operator-homotopies), and construct the semigroup $\overline{\UKK}(A,B)$. In \cref{sec:functional_dampening} we show that the notion of functional dampening can be implemented via an operator-homotopy. In \cref{sec:lift_homotopy} we construct the lift of a homotopy and prove \cref{main:lift} (see \cref{thm:lift_homotopy}). Combined with the operator-homotopy obtained from functional dampening, we then obtain \cref{main:isomorphism} (see \cref{thm:bdd_transform_KK_isom}). 

We introduce our notions of algebraically and spectrally degenerate cycles in \cref{sec:degenerate}, and we prove that degenerate cycles are null-homotopic (\cref{lem:alg_deg_null,prop:deg_hom_0}). In \cref{sec:op-hom_mod_null} we then show that any homotopy can be implemented as an operator-homotopy modulo addition of degenerate cycles (\cref{thm:lift_op-hom_mod_null}), which completes the proof of \cref{main:op-hom_mod_null}. 

We give a direct proof that $\overline{\UKK}(A,B)$ is a group (and not just a semigroup) in \cref{sec:symmetries}. In the case where $A$ is separable, this follows immediately from the isomorphism $\overline{\UKK}(A,B) \simeq \KK(A,B)$, but our direct proof works for any pair $(A,B)$ of $\sigma$-unital $C^{*}$-algebras. 
The proof relies on the observation that the presence of certain symmetries induces homotopical triviality. After a brief discussion of Lipschitz regular cycles in \cref{sec:Lip_reg}, we introduce the notion of spectrally symmetric cycles in \cref{sec:spec_symm}. These cycles are a mild generalisation of the notion of spectrally decomposable modules introduced in \cite{Kaa19pre}. We prove that any spectrally symmetric cycle is a bounded perturbation of a spectrally degenerate cycle, and therefore null-homotopic. 
In \cref{sec:Cliff_symm} we introduce the notion of Clifford symmetric cycles, which are elements of $\overline{\Psi}_1(A,B)$ which extend to $\overline{\Psi}_1(A\hot\CCliff_1,B)$. We prove that every Clifford symmetric cycle is operator-homotopic to a spectrally symmetric cycle and therefore null-homotopic. The proof is easily generalised to show that in fact every weakly degenerate cycle is null-homotopic. 
We exploit such Clifford symmetries to prove in \cref{sec:group} that the semigroup $\overline{\UKK}(A,B)$ is in fact a group. 

Finally, \cref{sec:appendix} contains some basic facts regarding localisations of Hilbert $C^*$-modules and their dense submodules.

\subsubsection*{Acknowledgements}
This work was initiated during the thematic program \emph{Bivariant K-theory in Geometry and Physics}, Vienna, November 2018, and the authors thank the organisers, as well as the Erwin Schr\"odinger Institute (ESI), for their hospitality. 
The authors also thank the Hausdorff Center for Mathematics in Bonn for hospitality and support. BM thanks the Max Planck Institute for Mathematics in Bonn. We thank Magnus Goffeng and Jens Kaad for inspiring conversations. We also thank Magnus Goffeng for a careful proofreading of this paper and Robin Deeley for his comments on the manuscript.

\subsubsection*{Notation and conventions}

Let $A$ and $B$ denote $\sigma$-unital $\Z_2$-graded $C^*$-algebras. 
By an approximate unit for $A$ we will always mean an even, positive, increasing, and contractive approximate unit for the $C^{*}$-algebra $A$. For elements $a,b\in A$ we denote by $[a,b]$ the graded commutator. If $a$ and $b$ are homogenous, we denote by $\deg a,\deg b\in \mathbb{Z}_2$ their degree and $[a,b]:=ab-(-1)^{\deg a\deg b}ba$. For general $a,b$ we extend the graded commutator by linearity. 
Let $E$ be a $\Z_2$-graded Hilbert $C^{*}$-module over $B$, or Hilbert $B$-module for short (for definitions and further details regarding Hilbert $C^*$-modules, we refer to the books \cite{Lance95,Blackadar98}). Throughout this article, we will assume $E$ is countably generated. 
We write $\End^{*}_{B}(E)$ for the adjointable operators on $E$, and $\End^{0}_{B}(E)$ for the compact operators on $E$. 
For any subset $W\subset\End_B^*(E)$, we write $\overline{W}$ for the closure of $W$ with respect to the operator-norm of $\End^*_B(E)$.

\section{Unbounded cycles}
\label{sec:Kasmod}

Kasparov \cite{Kas80} defined the abelian group $\KK(A,B)$ as a set of homotopy equivalence classes of Kasparov $A$-$B$-modules. 
We briefly recall the main definitions (more details can be found in e.g.\ \cite[\S17]{Blackadar98}). 

A (bounded) \emph{Kasparov $A$-$B$-module} is a triple $(\pi,E,F)$ consisting of a $\Z_2$-graded, countably generated, right Hilbert $B$-module $E$, a ($\Z_2$-graded) $*$-homomorphism $\pi\colon A\to\End_B^*(E)$, and an odd adjointable endomorphism $F\in\End_B^*(E)$ such that for all $a\in A$: 
\[\pi(a)(F-F^*), \quad [F,\pi(a)], \quad \pi(a)(F^2-1)\in\End^{0}_{B}(E).\] 
Two Kasparov $A$-$B$-modules $(\pi_0,E_0,F_0)$ and $(\pi_1,E_1,F_1)$ are called \emph{unitarily equivalent} (denoted with $\simeq$) if there exists an even unitary in $\Hom_B(E_0,E_1)$ intertwining the $\pi_j$ and $F_j$ (for $j=0,1$). 
A \emph{homotopy} between $(\pi_0,E_0,F_0)$ and $(\pi_1,E_1,F_1)$ is given by a Kasparov $A$-$C([0,1],B)$-module $(\til\pi,\til E,\til F)$ such that \[\ev_j(\til\pi,\til E,\til F) \simeq (\pi_j,E_j,F_j), \quad j=0,1.\] 
A homotopy $(\til\pi,\til E,\til F)$ is called an \emph{operator-homotopy} if there exists a Hilbert $B$-module $E$ with a representation $\pi\colon A\to\End_B^*(E)$ such that $\til E$ equals the Hilbert $C([0,1],B)$-module $C([0,1],E)$ with the natural representation $\til\pi$ of $A$ on $C([0,1],E)$ induced from $\pi$, and if $\til F$ is given by a \emph{norm}-continuous family $\{F_t\}_{t\in[0,1]}$. 
A module $(\pi,E,F)$ is called \emph{degenerate} if $\pi(a)(F-F^*) = [F,\pi(a)] = \pi(a)(F^2-1) = 0$ for all $a\in A$. 

The $\KK$-theory $\KK(A,B)$ of $A$ and $B$ is defined as the set of homotopy equivalence classes of (bounded) Kasparov $A$-$B$-modules. 
Since homotopy equivalence respects direct sums, the direct sum of Kasparov $A$-$B$-modules induces a (commutative and associative) binary operation (`addition') on the elements of $\KK(A,B)$ such that $\KK(A,B)$ is in fact an abelian group \cite[\S4, Theorem 1]{Kas80}. 

In this paper we will give a completely analogous description of $\KK$-theory, based instead on unbounded Kasparov modules \cite{BJ83}. 
Recall that a closed densely defined symmetric operator $\D\colon\Dom \D\to E$ is self-adjoint and regular if the operators $\D\pm i:\Dom \D \to E$ have dense range. We refer to \cite[Chapter 9 and 10]{Lance95} for details on regular operators on Hilbert modules.
For a self-adjoint regular operator $\D\colon\Dom \D\to E$, we write
\[
\Lip(\D) := \big\{ T\in\End^{*}_{B}(E): T(\Dom \D)\subset\Dom\D \; \& \; [\D,T]\in \End^{*}_{B}(E) \big\} . 
\]
It is worth noting that, because $\D$ is densely defined, $\overline{\Lip(\D)}\cap\End^{0}_{B}(E)$ is equal to $\End^{0}_{B}(E)$.
However, in general $\overline{\Lip(\D)}$ is not equal to $\End^{*}_{B}(E)$. 
We also introduce 
\[
\Lip^0(\D) := \big\{ T \in \Lip(\D) : T (1+\D^2)^{-\frac12} , T^* (1+\D^2)^{-\frac12} \in \End_B^0(E) \big\} . 
\]
We note that $\Lip^0(\D)$ is a $*$-subalgebra of $\End_B^*(E)$. 
We introduce the following relaxation of the notion of unbounded Kasparov module.
\begin{defn}
\label{def: Kasmod}
An \emph{unbounded $A$-$B$-cycle} $(\pi,E,\D)$ consists of a $\Z_2$-graded, countably generated Hilbert $B$-module $E$, a $\Z_2$-graded $*$-homomorphism $\pi \colon A \to \End_B(E)$, and an odd regular self-adjoint operator $\D$ on $E$, 
such that 
\[
\pi(A) \subset \overline{\Lip^0(\D)} .
\]
The set of all unbounded $A$-$B$-cycles is denoted $\overline{\Psi}_1(A,B)$. 
We will often suppress the representation $\pi$ in our notation and simply write $(E,\D)$ instead of $(\pi,E,\D)$. 
\end{defn}

\begin{remark}
\label{remark:Kasmod}
\begin{enumerate}
\item 
It follows immediately from the definition that $\pi(a) (1+\D^2)^{-\frac12} \in \End_B^0(E)$ for any $a\in A$, i.e.\ $\D$ has `$A$-locally compact' resolvents. 

\item 
We point out that if $\pi(A) \subset \End_B^0(E)$ (i.e.\ $A$ is represented as compact operators), then the condition $\pi(A) \subset \overline{\Lip^0(\D)}$ is automatically satisfied, since $\Lip(\D)\cap\End_B^0(E) \subset \Lip^0(\D)$ is always dense in $\End_B^0(E)$. 
\end{enumerate}
\end{remark}

\begin{remark}
\label{rmk: Kasmod_ordinary}
We use the term unbounded $A$-$B$-cycle since our definition is different from the usual definition of an unbounded Kasparov module, originally given in \cite{BJ83} (see also \cite[17.11.1]{Blackadar98}). 
An unbounded $A$-$B$-cycle $(\pi,E,\D)$ is an unbounded Kasparov module if there exists a dense $*$-subalgebra $\A\subset A$ such that $\pi(\A) \subset \Lip^0(\D)$. 
To avoid confusion we often refer to such cycles as \emph{ordinary unbounded Kasparov modules.}
\end{remark}
Our main reason for relaxing this definition is the following simple lemma. 
\begin{lem}
\label{lem:semigroup}
The direct sum of unbounded $A$-$B$-cycles is well-defined, and therefore $\overline{\Psi}_1(A,B)$ is a semigroup. 
\end{lem}
\begin{proof} 
Given unbounded  $A$-$B$-cycles $(\pi_{i}, E_{i},\D_{i})$, $i=0,1$, we have $\Lip^0(\D_0)\oplus\Lip^0(\D_1)\subset\Lip^0(\D_0\oplus \D_1)$ and $\pi_i(A) \subset \overline{\Lip^0(\D_i)}$. It follows that
\[(\pi_0\oplus\pi_1)(A)\subset \overline{\Lip^0(\D_0)\oplus\Lip^0(\D_1)}\subset\overline{\Lip^0(\D_0\oplus \D_1)},\] 
and therefore $(\pi_{0}\oplus \pi_{1}, E_{0}\oplus E_{1}, \D_{0}\oplus \D_{1})$ is also an unbounded $A$-$B$-cycle. 
\end{proof}
\begin{remark}
\label{remark:subalgebra}
Note that if there are dense $*$-subalgebras $\A_{i}\subset A$ such that $\pi_i(\A_i)\subset \Lip(\D_i)$, it may not be possible to find a dense $*$-subalgebra $\A\subset A$ such that $$(\pi_0\oplus\pi_1)(\A) \subset \Lip^0(\D_0\oplus\D_1).$$ 
In fact, even if $E_0=E_1$ and $\pi_0=\pi_1=\pi$, the intersection $$\Lip(\D_0)\cap\Lip(\D_1)\cap\pi(A)$$ might not be dense in $\pi(A)$ (for an example, see for instance \cite[Appendix A]{DGM18}). Hence the direct sum is not well-defined on \emph{ordinary} unbounded Kasparov modules. 
The usual way around this problem is to \emph{fix} a dense $*$-subalgebra $\A\subset A$, and to consider only those unbounded Kasparov modules $(\pi,E,\D)$ for which $\pi(\A)\subset\Lip^0(\D)$. 
With our relaxed condition $\pi(A) \subset \overline{\Lip^0(\D)}$, we avoid the need to make such a choice for $\A$.
\end{remark}

\begin{lem}
\label{lem:separable_Lip}
Let $(\pi,E,\D)$ be an unbounded $A$-$B$-cycle, and suppose that $A$ is separable. 
Then there exists a countable subset $W \subset \Lip^0(\D)$ consisting of products of elements in $\Lip^0(\D)$ (i.e.\ each $w\in W$ is of the form $w=T_1T_2$ for $T_1,T_2\in\Lip^0(\D)$) such that $\pi(A) \subset \overline{W}$. 
\end{lem}
\begin{proof}
Since $A$ is separable, and since products are dense in any $C^*$-algebra, we may pick a countable dense subset of products $\{a_j b_j\}_{j\in\N} \subset A$. Since $\pi(A) \subset \overline{\Lip^0(\D)}$, there exist sequences $\{v_{j,k}\}_{k\in\N} , \{w_{j,k}\}_{k\in\N} \subset \Lip^0(\D)$ such that $$\lim_{k}\|a_j-v_{j,k}\| = \lim_{k}\|b_j-w_{j,k}\| = 0.$$ 
The statement then holds with $W := \{ v_{j,k} w_{j,k} \}_{j,k\in\N}$. 
\end{proof}

Baaj and Julg proved for any ordinary unbounded Kasparov module that the bounded transform $\D \mapsto F_\D:= D(1+\D^2)^{-\frac12}$ yields a bounded Kasparov module and hence a $\KK$-class. 
Before we continue, we need to show that this still holds for our relaxed definition of unbounded cycles. 

\begin{prop}[{cf.\ \cite{BJ83}}]
\label{prop:bdd_transform}
If $(\pi,E,\D)$ is an unbounded  $A$-$B$-cycle (as in \cref{def: Kasmod}), then the bounded transform $(\pi, E, F_\D)$ is a bounded Kasparov module and hence defines an element in $\KK(A,B)$.
\end{prop}
\begin{proof}
As in \cite[Proposition 17.11.3]{Blackadar98}, it suffices to show that $[F_\D,a]b$ is compact for any $a,b\in A$. By \cref{def: Kasmod}, there is a sequence $T_{n}\in\Lip^0(\D)$ such that $T_{n}\to a$ in norm, and then $[F_{\D},T_{n}] b \to [F_{\D},a]b$ in norm as well. It thus suffices to show that $[F,T]b\in\End^0_B(E)$ for $b\in A$ and $T\in\Lip^0(\D)$. 
But compactness of $[F,T]b$ follows from exactly the same argument as in \cite[Proposition 17.11.3]{Blackadar98}. 
\end{proof}

\subsection{The algebras \texorpdfstring{$C_F$ and $J_F$}{C and J}}
\label{sec:C_J}

Let $E$ be a countably generated Hilbert $B$-module. 
The following result is well known, and follows from the proof of \cite[Proposition 13.6.1]{Blackadar98} (which extends from $h\in\End_B^0(E)$ to arbitrary $h\in\End_B^*(E)$). 
\begin{lem}[{cf.\ \cite[Proposition 13.6.1]{Blackadar98}}]
\label{lem:dense_range_cpts}
Let $h\in\End_B^{*}(E)$. Then $hE$ is dense in $E$ if and only if $h\cdot\End_B^0(E)$ is dense in $\End_B^0(E)$. 
\end{lem}

For a bounded Kasparov $A$-$B$-module $(E,F)$ with $F=F^{*}$ and $F^{2}\leq 1$, we define
\begin{align*}
C_F &:= C^{*}(1-F^{2}) + F C^{*}(1-F^{2}) , & 
J_{F} &:= \End^{0}_{B}(E) + C_F . 
\end{align*}
The $C^{*}$-algebra $J_{F}$ was introduced in \cite[Lemma 4.5]{MR16}, and plays an important role in the construction of the (unbounded) lift of a (bounded) Kasparov module. 
\begin{lem}
\label{lem:C_F}
The space $C_F$ is a separable $C^*$-algebra, and $1-F^2$ is a strictly positive element in $C_F$. 
\end{lem}
\begin{proof}
It is explained in the proof of \cite[Lemma 4.5]{MR16} that $C_F$ is a separable $C^*$-algebra. 
By assumption, the spectrum $\spec(F)$ of $F$ is contained in $[-1,1]$, and by construction $C_F$ can be identified with a $*$-subalgebra of $C_0(\spec(F)\backslash\{\pm1\})$. Because $C_F$ vanishes nowhere and separates points of $\spec(F)\backslash\{\pm1\}$, the Stone-Weierstrass theorem implies that $C_F \simeq C_0(\spec(F)\backslash\{\pm1\})$. Since $x \mapsto 1-x^2$ is a strictly positive function on $C_0(\spec(F)\backslash\{\pm1\})$, it follows that $1-F^2$ is a strictly positive element in $C_F$. 
\end{proof}

\begin{lem}
\label{lem:J_F}
The space $J_F$ is a $\sigma$-unital $C^*$-algebra, and we have the inclusions 
$$A J_{F},\, J_{F}A,\, F J_{F},\, J_{F}F \subset J_{F}.$$ 
Furthermore, if $k\in\End_B^0(E)$ is a positive operator such that $k + (1-F^2)$ has dense range in $E$, then $k + (1-F^2)$ is strictly positive in $J_F$. 
\end{lem}
\begin{proof}
As $E$ is countably generated, $\End_B^0(E)$ is a $\sigma$-unital $C^*$-algebra (see e.g.\ \cite[Proposition 6.7]{Lance95}). 
Since $\End_B^0(E)$ is an ideal in $\End_B^*(E)$, it follows from \cite[\S3, Lemma 2]{Kas80} that $J_F$ is also a $\sigma$-unital $C^*$-algebra. The inclusions $F J_{F}, J_{F}F \subset J_{F}$ are immediate, and the inclusions $A J_{F}, J_{F}A\subset J_{F}$ follow because $a(1-F^2)$ and $[F,a]$ are compact for all $a\in A$. 

Let $k\in\End_B^0(E)$ be a positive operator such that $h := k + (1-F^2)$ has dense range in $E$. 
Consider an element $l+c \in J_F$ where $l\in\End_B^0(E)$ and $c\in C_F$, and let $\varepsilon>0$. 
Since $1-F^2$ is strictly positive in $C_F$ by \cref{lem:C_F}, there exists $b\in C_F$ such that $\|(1-F^2)b-c\| < \varepsilon$. Moreover, since $l-kb$ is compact, we know from \cref{lem:dense_range_cpts} that there exists $a\in\End_B^0(E)$ such that $\|ha - (l-kb)\| < \varepsilon$. Hence 
\[
\big\| h (a+b) - (l+c) \big\| \leq \big\| ha - (l-kb) \big\| + \big\| (1-F^2) b - c \big\| < 2\varepsilon ,
\]
which proves that $h J_F$ is dense in $J_F$. 
\end{proof}

\subsection{The lifting construction}
\label{sec:lift}

Since our definition of unbounded cycle is more general than the usual definition of unbounded Kasparov module, it of course remains true that the bounded transform is surjective \cite{BJ83}. 
The way to prove this surjectivity is by showing that every bounded Kasparov module $(E,F)$ can be lifted to an (ordinary) unbounded Kasparov module $(E,\D)$ such that $F_\D$ is operator-homotopic to $F$.
Because we will make essential use of the technical subtleties of this lifting procedure in the sequel, we present the proof here, closely following the arguments of \cite{MR16,Kuc00}. 
Recall that all approximate units are assumed to be even, positive, increasing, and contractive for the $C^{*}$-algebra norm.
\begin{lem}[{cf.\ \cite[proof of Theorem 1.25]{MR16}}] 
\label{lem:lift}
Let $C$ be a commutative separable $C^{*}$-algebra, $\{c_{j}\}_{j\in\mathbb{N}}\subset C$ a total subset, and $\{u_{n}\}_{n\in\N}$ a countable commutative approximate unit for $C$. 
If for some $0<\varepsilon <1$, $d_{n}:=u_{n+1}-u_{n}$ satisfies 
$$ 
\|d_{n}c_{j}\|\leq \varepsilon^{2n} , \quad \forall j\leq n ,
$$ 
then the series $l^{-1}:=\sum \varepsilon^{-n}d_{n}$ defines an unbounded multiplier on $C$ such that $l:=(l^{-1})^{-1}\in C$ is strictly positive.
\end{lem}
\begin{proof}
The series $l^{-1}c_{j}:=\sum_{n} \varepsilon^{-n}d_{n}c_{j}$ is convergent for all $j$ by our assumption that $\|d_n c_j\|\leq \varepsilon^{2n}$ for all $n\geq j$, so $l^{-1}$ is a densely defined unbounded multiplier. 
The partial sums $\sum_{n=0}^{k}\varepsilon^{-n}d_{n}$ are elements in the commutative $C^{*}$-algebra $C\simeq C_{0}(Y)$, where $Y=\textnormal{Spec } C$. Under this isomorphism, the approximate unit $u_{n}$ is identified with a sequence of functions converging pointwise to $1$. For fixed $t\in (0,1)$ set
\[Y_{k}:=\{y\in Y: u_{k}(y)\geq t\},\]
which gives an increasing sequence of compact sets $Y_{k}\subset Y_{k+1}$ with $\bigcup_{k=0}^{\infty}Y_{k}=Y$. Let $y\in Y\setminus Y_{k}$ and $m\geq k$.
We have the estimates 
\begin{align*}
\sum_{n=0}^{\infty}\varepsilon^{-n}d_{n}(y)&\geq \sum_{n=k}^{n=m}\varepsilon^{-n}d_{n}(y)+\sum_{n=m+1}^{\infty}\varepsilon^{-n}d_{n}(y)\\
&\geq \varepsilon^{-k}(u_{m+1}-u_{k})(y)+\sum_{n=m+1}^{\infty}\varepsilon^{-n}d_{n}(y)\\
&\geq \varepsilon^{-k}(u_{m+1}(y)-t)+\sum_{n=m+1}^{\infty}\varepsilon^{-n}d_{n}(y)\to \varepsilon^{-k}(1-t) , 
\end{align*}
as $m\to \infty$. This shows that $l^{-1}$ is given by a function whose reciprocal is a strictly positive function in $C_{0}(Y)$, so this defines a strictly positive element $l\in C$. 
\end{proof}

\begin{prop} 
\label{prop:lift}
Let $(E,F)$ be a bounded Kasparov $A$-$B$-module satisfying $F^{*}=F$ and $F^{2}\leq 1$. 
Given a countable dense subset $\A\subset A$, there exists a positive operator $l\in J_{F}$ with dense range in $E$ such that 
\begin{enumerate}
\item the (closure of the) operator $\D:=\frac{1}{2}(Fl^{-1}+l^{-1}F)$ makes $(E,\D)$ into an ordinary unbounded Kasparov $A$-$B$-module with $\A\subset \Lip^{0}(\D)$; 
\item $F$ and $F_{\D}$ are operator-homotopic. 
\end{enumerate}
Moreover, if $F^{2}=1$, we can ensure that $l$ commutes with $F$ and that $(1+\D^2)^{-\frac12}$ is compact. 
\end{prop}
\begin{proof}
Pick an even strictly positive element $h\in J_{F}$. Since we have (cf.\ \cref{lem:J_F})
$$A J_{F},\, J_{F}A,\, F J_{F},\, J_{F}F \subset J_{F},$$ 
there exists by \cite[Theorem 3.2]{AkPed} an approximate unit $u_{n}\in C^{*}(h)$ for $J_F$ that is quasicentral for $A$ and $F$.
Let $\{a_i\}_{i\in\mathbb{N}}$ be an enumeration of $\A$, choose a countable dense subset $\{c_i\}_{i\in\N} \subset C^*(h)$, and fix a choice of $0<\varepsilon<1$. 
By selecting a suitable subsequence of $u_{n}$, we can furthermore achieve that, for each $n\in\N$, $d_{n}:=u_{n+1}-u_{n}$ satisfies
\begin{enumerate}[(a)]
\item $\|d_{n}c_{i}\|\leq \varepsilon^{2n}$ for all $i\leq n$;
\item $\|d_{n}(1-F^{2})^{\frac{1}{4}}\| \leq \varepsilon^{2n}$;
\item $\|d_{n}[F,a_{i}]\|\leq \varepsilon^{2n}$ for all $i\leq n$;
\item $\|[d_{n},a_{i}]\|\leq \varepsilon^{2n}$ for all $i\leq n$;
\item $\|[d_{n},F]\|\leq \varepsilon^{2n}$.
\end{enumerate}
Here properties (a)-(c) follow because $u_n$ is an approximate unit for $J_F$ (and $c_i$, $(1-F^{2})^{\frac{1}{4}}$, and $[F,a_{i}]$ all lie in $J_F$), and properties (d)-(e) follow because $u_n$ is quasicentral for $A$ and $F$. 
By property (a) and \cref{lem:lift} we obtain a strictly positive element $l\in C^{*}(h)$ such that $l^{-1}=\sum \varepsilon^{-n}d_{n}$. 
Since $lJ_{F}\supset lC^{*}(h)J_{F}$, $lC^*(h)$ is dense in $C^*(h)$, and $C^*(h) J_F$ is dense in $J_{F}$, it follows that $lJ_F$ is dense in $J_F$ and therefore $l$ is strictly positive in $J_F$. In particular, $l$ has dense range in $E$. 
From properties (b)-(e) it follows that $ l \in C^{*}(h)\subset J_{F}$ satisfies \cite[Definition 4.6]{MR16}. 
Then by \cite[Theorem 4.7]{MR16} and \cite[Lemma 2.2]{Kuc00} the (closure of the) operator 
$$\mathcal{D} := \frac12(F  l^{-1} +  l^{-1} F)$$ is a densely defined and regular self-adjoint operator on $E$, and $(E,\mathcal{D})$ is an ordinary unbounded Kasparov $A$-$B$-module with $\A\subset\Lip^0(\D)$. Furthermore, the proof of \cite[Theorem 4.7]{MR16} (combined with \cite[Proposition 17.2.7]{Blackadar98}) shows that $F_{\D}$ is operator-homotopic to $F$.

For the final statement, suppose $F^{2}=1$, so that $J_F = \End_B^0(E)$. For any positive element $k\in J_{F}$ with dense range, we can consider $h:=k+FkF\geq k$, which is also positive with dense range (see \cite[Corollary 10.2]{Lance95}). 
Then $h$ is a strictly positive element in $\End_B^0(E)$ (see \cref{lem:dense_range_cpts}), and $h$ commutes with $F$. We then proceed as above (conditions (b) and (e) now being redundant) to construct a compact operator $l\in C^{*}(h)$ which also commutes with $F$. Lastly, for $\D = Fl^{-1}$ we see that $(1+\D^2)^{-\frac12} = l (1+l^2)^{-1}$ is indeed compact. 
\end{proof}

\cref{prop:lift} immediately implies the surjectivity of the bounded transform: 
\begin{thm}[\cite{BJ83, Kuc00}, cf.\ {\cite[Theorem 17.11.4]{Blackadar98}}]
\label{thm:bdd_transform_surj}
If $A$ is separable, then the bounded transform gives a surjective map $\overline{\Psi}_1(A,B) \to \KK(A,B)$.
\end{thm}

\subsection{Regular self-adjoint operators}
\label{sec:reg_sa}

Let $\D$ be a regular self-adjoint operator on a Hilbert $B$-module $E$. We recall from \cite[Theorem 10.9]{Lance95} that there exists a continuous functional calculus for $\D$, i.e.\ a $*$-homomorphism $f\mapsto f(\D)$ from $C(\R)$ to the regular operators on $E$, such that $\id(\D) = \D$ and $b(\D) = F_\D := \D(1+\D^2)^{-\frac12}$ (where $b(x) = x(1+x^2)^{-\frac12}$). 
In particular, if $f\in C(\R)$ is real-valued, then $f(\D)$ is regular self-adjoint. 

If the operators $a (\D\pm i)^{-1}$ are compact for some $a\in\End_B^*(E)$, we note that also $a g(\D)$ is compact for any $g\in C_0(\R)$ (since the functions $x\mapsto(x\pm i)^{-1}$ generate $C_0(\R)$). In particular, if $f\in C(\R)$ is a real-valued function such that $\lim_{x\to\pm\infty} |f(x)| = \infty$, then $a(f(\D)\pm i)^{-1}$ and $a (1+f(\D)^2)^{-\frac12}$ are compact. 

For completeness, we will show that the continuous functional calculus is compatible with $\Z_2$-gradings. 

\begin{lem}
\label{lem:odd_funct_calc}
Let $\D$ be an odd regular self-adjoint operator on a $\Z_2$-graded Hilbert $B$-module $E$.  If $f\in C(\R)$ is an odd real-valued function, then the regular self-adjoint operator $f(\D)$ is also odd.
\end{lem}
\begin{proof}
Let $\Gamma$ denote the $\Z_2$-grading operator on $E$, and let us grade $C_0(\R)$ by even and odd functions. 
As in the proof of \cite[Lemma 10.6.2]{Higson-Roe00}, the identity 
\[\Gamma (i\pm\D)^{-1} = (i\mp\D)^{-1} \Gamma,\]
shows that $\Gamma$ graded-commutes with $(i\pm\D)^{-1}$ and hence with any element in $C_0(\R)$. 
The linear subspace $\E := \{ g(\D)\psi \mvert g\in C_c(\R) , \, \psi\in E\}$ is a core for $f(\D)$ (cf.\ \cite[Lemma 10.8]{Lance95}). 
Each $g\in C_c(\R)$ is the sum of an even function $g_0\in C_c(\R)$ and an odd function $g_1\in C_c(\R)$. 
Then we have $\Gamma\E \subset \E$. 
Moreover, since $fg_0\in C_c(\R)$ is odd and $fg_1\in C_c(\R)$ is even, we find that 
\[
\Gamma f(\D) g(\D) = - f(\D) g_0(\D) \Gamma + f(\D) g_1(\D) \Gamma = - f(\D) \Gamma g(\D) .
\]
Thus $[f(\D),\Gamma]_+ = 0$ on the core $\E$, and it follows that in fact $\Gamma$ preserves $\Dom f(\D)$ and $f(\D)$ anti-commutes with $\Gamma$. 
\end{proof}

\begin{lem}
\label{lem:fam_reg_sa}
Let $X$ be a locally compact Hausdorff space and $Y\subset X$ an open subset. 
Let $\{\D_y\}_{y\in Y}$ be a family of regular self-adjoint operators on a Hilbert $B$-module $E$, and assume there exists a dense submodule $\E\subset E$ which is a core for $\D_y$ for each $y\in Y$, such that for each $\psi\in\E$ the map $Y\to E$, $y\mapsto \D_y\psi$ is continuous. 
Then the operator $\til\D$ on the Hilbert $C_0(X,B)$-module $C_0(Y,E)$ defined by 
\begin{align*}
\Dom\til\D &:= \big\{ \psi\in C_0(Y,E) : \psi(y)\in\Dom\D_{y} , \; \til\D\psi\in C_0(Y,E) \big\} , & 
(\til\D\psi)(y) &:= \D_y\psi(y) .
\end{align*}
is regular and self-adjoint. 
\end{lem}
\begin{proof}
Consider the algebraic tensor product $\til\E := C_c(Y)\otimes\E$. 
Since $y\mapsto \D_y\psi$ is continuous for each $\psi\in\E$, we note that $\til\E \subset \Dom\til\D$. 
In particular, since $\til\E$ is dense in $C_0(Y,E)$, we know that $\til\D$ is densely defined. Moreover, since $\D_y$ is closed on $\Dom\D_y$, it follows that also $\til\D$ is closed on $\Dom\til\D$. 
By assumption, the operators $\D_y\pm i:\E\to E$ have dense range in $E$ for all $y\in Y$. 
Since $C_0(Y,E)\hot_{\ev_x}B=\{0\}$ for $x\notin Y$, it follows from \cref{locallydenserange} that the operators $\til{\D}\pm i\colon \til\E\to C_0(Y,E)$ have dense range in $C_0(Y,E)$, and therefore $\til\D$ is regular and self-adjoint. 
\end{proof}

\begin{remark}
We will apply the above lemma to construct operator-homotopies over $X = [0,1]$, and the two main cases of interest are $Y=X$ or $Y=(0,1]$. 
\end{remark}

\section{The unbounded homotopy relation}
\label{sec:unbdd_homotopy}

\subsection{The homotopy semigroup}
\label{sec:semigroup}

For any $t\in[0,1]$, we have the surjective $*$-homomorphism $\ev_t\colon C([0,1],B) \to B$ given by $\ev_t(b) := b(t)$. 
Given an unbounded Kasparov $A$-$C([0,1],B)$-cycle $(\pi,E,\D)$, we then define $\ev_t(\pi,E,\D)=(\pi_{t},E_{t},\D_{t}) := (\pi\hot1, E\hot_{\ev_t}B,\D\hot1)$. 

\begin{defn}
\label{defn:unbdd_KK_equiv}
Consider unbounded $A$-$B$-cycles $(\pi_0,E_0,\D_0)$ and $(\pi_1,E_1,\D_1)$. 
We introduce the following notions:
\begin{description}
\item[Unitary equivalence:] $(\pi_0,E_0,\D_0)$ and $(\pi_1,E_1,\D_1)$ are called \emph{unitarily equivalent} (denoted $(\pi_0,E_0,\D_0) \simeq (\pi_1,E_1,\D_1)$) if there exists an even unitary $U\colon E_0\to E_1$ such that $U\D_0 = \D_1 U$ and $U \pi_0(a) = \pi_1(a) U$ for all $a\in A$. 

\item[Homotopy:] 
A homotopy between $(\pi_0,E_0,\D_0)$ and $(\pi_1,E_1,\D_1)$ is given by an unbounded 
$A$-$C([0,1],B)$-cycle $(\til\pi,\til E,\til\D)$ such that $\ev_j(\til\pi,\til E,\til\D) \simeq (\pi_j,E_j,\D_j)$ for $j=0,1$. 

\item[Operator-homotopy:] 
A homotopy $(\til\pi,\til E,\til\D)$ is called an \emph{operator-homotopy} if there exists a Hilbert $B$-module $E$ with a representation $\pi\colon A\to\End_B^*(E)$ such that $\til E$ equals the Hilbert $C([0,1],B)$-module $C([0,1],E)$ with the natural representation $\til\pi$ of $A$ on $C([0,1],E)$ induced from $\pi$.
\end{description}
We denote by $\sim_{oh}$ the equivalence relation on $\overline{\Psi}_1(A,B)$ generated by operator-homotopies and unitary equivalences. 
The homotopy relation is denoted $\sim_h$. 
\end{defn}

\begin{remark}
If $(\pi,E,\D)$ is an unbounded  $A$-$B$-cycle such that $\pi(A) \subset \End_B^0(E)$ (i.e.\ $A$ is represented as compact operators), then $(E,\D)$ is operator-homotopic to $(\pi,E,0)$, via the operator-homotopy given by $\D_t = t\D$ for $t\in[0,1]$ (see also \cref{remark:Kasmod}.(2)). 
\end{remark}

We note that it was shown in \cite[Proposition 4.6]{Kaa19pre} that the homotopy relation is an equivalence relation on unbounded Kasparov modules. 
We will show next that the proof extends to our more general notion of unbounded cycles from \cref{def: Kasmod}, and for this purpose we recall some notation from \cite[\S4]{Kaa19pre}. 
Consider two unbounded $A$-$C([0,1],B)$-cycles $(\pi,E,\D)$ and $(\pi',E',\D')$, and a unitary isomorphism $U \colon E\hot_{\ev_1} B \to E'\hot_{\ev_0}B$ satisfying 
\begin{align*}
U \big( \pi(a)\hot_{\ev_1}1 \big) U^* &= \pi'(a)\hot_{\ev_0}1 , & 
U \big( \D\hot_{\ev_1}1 \big) U^* &= \D'\hot_{\ev_0}1 , 
\end{align*}
for any $a\in A$. 
For $t\in[0,1]$ we consider the localisations $E_t := E\hot_{\ev_t}B$, and for $e\in E$ we write $e_t := e\hot_{\ev_t}1 \in E_t$ (as in the Appendix). 
We define the \emph{concatenation}
\begin{align*}
E \times_U E' := \big\{ (e,e') \in E\oplus E' : Ue_1 = e'_0 \big\} .
\end{align*}
The space $E \times_U E'$ is endowed with the right action of $C([0,1],B)$ and the inner product described in \cite[\S4]{Kaa19pre}. 
We note that $\pi\oplus\pi'$ and $\D\oplus\D'$ are well-defined on $E \times_U E'$, and that $\D\oplus\D'$ is a regular self-adjoint operator (see the proof of \cite[Proposition 4.6]{Kaa19pre}). 
For two linear subspaces $W\subset\End_{C([0,1],B)}(E)$ and $W'\subset\End_{C([0,1],B)}(E')$, we write
\[
W \times_U W' := \big\{ (w,w')\in W\oplus W' : U(w\hot_{\ev_1}1)U^* = w'\hot_{\ev_0}1 \big\} .
\]
We note that we have the inclusion $\Lip(\D) \times_U \Lip(\D') \subset \Lip(\D\oplus\D')$. In fact, using \cite[Lemma 4.5]{Kaa19pre}, we obtain 
\[
\Lip^0(\D) \times_U \Lip^0(\D') \subset \Lip^0(\D\oplus\D') . 
\]

\begin{prop}[cf.\ {\cite[Proposition 4.6]{Kaa19pre}}]
The homotopy relation on unbounded $A$-$B$-cycles is an equivalence relation. 
\end{prop}
\begin{proof}
Reflexivity and symmetry are proven exactly as in \cite[Proposition 4.6]{Kaa19pre}. For transitivity, we need to show that the concatenation of two unbounded $A$-$C([0,1],B)$-cycles is again an unbounded $A$-$C([0,1],B)$-cycle. 

We will first show that we may assume (without loss of generality) that any unbounded $A$-$C([0,1],B)$-cycle $(\pi,E,\D)$ is `constant near the endpoints'. 
We define 
\begin{align*}
\til E &:= C([0,1],E_0) \times_{\Id} E , &
\til\pi(a) &:= \pi_0(a) \oplus \pi(a) , & 
\til\D &:= \D_0\oplus\D . 
\end{align*}
Here $\pi_0(a)$ and $\D_0$ denote the obvious extension to $C([0,1],E_0)$ of the operators $\pi(a)\hot_{\ev_0}1$ and $\D\hot_{\ev_0}1$ on $E_0$, respectively. 
Now consider $\varepsilon>0$ and $a\in A$. 
Pick $S\in\Lip^0(\D)$ such that $\|\pi(a)-S\| < \varepsilon$. Then we also have $\|\pi_0(a)-S_0\| < \varepsilon$ and therefore $\|\til\pi(a)-S_0\oplus S\| < \varepsilon$. 
This proves that we have the inclusions 
\[
\til\pi(A) \subset \bar{\Lip^0(\D_0) \times_{\Id} \Lip^0(\D)} \subset \bar{\Lip^0(\til\D)} ,
\]
so $(\til\pi,\til E,\til\D)$ is an unbounded $A$-$C([0,1],B)$-cycle which is constant on $[0,\frac12]$. 

Now suppose we have two unbounded $A$-$C([0,1],B)$-cycles $(\pi,E,\D)$ and $(\pi',E',\D')$, and a unitary isomorphism $U \colon E\hot_{\ev_1} B \to E'\hot_{\ev_0}B$ satisfying 
\begin{align*}
U \big( \pi(a)\hot_{\ev_1}1 \big) U^* &= \pi'(a)\hot_{\ev_0}1 , & 
U \big( \D\hot_{\ev_1}1 \big) U^* &= \D'\hot_{\ev_0}1 , 
\end{align*}
for any $a\in A$. 
As described above, we may assume (without loss of generality) that $(\pi',E',\D')$ is constant on $[0,\frac12]$. 
We define
\begin{align*}
E'' &:= E \times_U E' , &
\pi''(a) &:= \pi(a) \oplus \pi'(a) , & 
\D'' &:= \D\oplus\D' . 
\end{align*}
Now consider $\varepsilon>0$ and $a\in A$. 
Pick $S\in\Lip^0(\D)$ such that $\|\pi(a)-S\| < \varepsilon$. Then in particular we have 
\[
\|\pi'(a)_0 - US_1U^*\| = \|\pi_1(a)-S_1\| < \varepsilon .
\]
Pick a function $\chi\in C^\infty([0,1])$ such that $0\leq\chi\leq1$, $\chi(0)=1$, and $\chi(t) = 0$ for all $\frac12\leq t\leq 1$. 
Since $E'$ is constant on $[0,\frac12]$, we note that $\chi US_1U^*$ is a well-defined adjointable operator on $E'$, which in fact lies in $\Lip^0(\D')$. 
If we also pick $R'\in\Lip^0(\D')$ such that $\|\pi'(a)-R'\| < \varepsilon$, then we obtain $T'' := S \oplus (\chi US_1U^*+(1-\chi) R') \in \Lip^0(\D) \times_U \Lip^0(\D')$ and we have the estimate 
\begin{align*}
\big\| \pi''(a) - T'' \big\| &\leq \max\big\{ \|\pi(a) - S\| , \|\pi'(a) - \chi US_1U^*+(1-\chi) R'\| \big\} \\
&\leq \max\big\{ \|\pi(a) - S\| , \sup_{t\in[0,1]} \big( \chi(t) \|\pi'_0(a)-US_1U^*\| + (1-\chi(t)) \|\pi'_t(a)-R'\| \big) \big\} \\
&< \varepsilon . 
\end{align*}
This proves that we have the inclusions 
\[
\pi''(A) \subset \bar{\Lip^0(\D) \times_U \Lip^0(\D')} \subset \bar{\Lip^0(\D'')} , 
\]
and we conclude that $(\pi'',E'',\D'')$ is again an unbounded $A$-$C([0,1],B)$-cycle. 
\end{proof}

\begin{defn}
We define $\overline{\UKK}(A,B)$ as the set of homotopy equivalence classes of unbounded $A$-$B$-cycles. 
\end{defn}
We recall from \cref{lem:semigroup} that the direct sum of two unbounded cycles is well-defined. 
Since the direct sum is also compatible with homotopies, we obtain a well-defined addition on $\overline{\UKK}(A,B)$ induced by the direct sum. 
Moreover, this addition is associative and commutative (since homotopy equivalence is weaker than unitary equivalence). 
Hence $\overline{\UKK}(A,B)$ is an \emph{abelian semigroup}, with the zero element given by the class of the zero cycle $(0,0)$.

\subsection{Functional dampening}
\label{sec:functional_dampening}

The goal of this subsection is to show that, up to operator-homotopy, we can replace an unbounded cycle $(E,\D)$ by $(E,f(\D))$ for suitable functions $f$ which blow up towards infinity at a sublinear rate.
One can think of $f(\D)$ as a `dampened' version of $\D$, and we refer to the transformation $\D \mapsto f(\D)$ as `functional dampening'. 
Our proof is partly inspired by the proof of \cite[Proposition 5.1]{Kaa19pre}, where the special case $f(x) := x(1+x^2)^{-r}$ (with $r\in(0,\frac12)$) is considered. 

\begin{defn}
A \emph{dampening function} is an odd continuous function $f\colon\R\to\R$ such that 
\begin{align*}
\lim_{x\to\infty} f(x) &= \infty , & 
\lim_{x\to\infty} f(x) (1+x^2)^{-\frac12} &= 0 . 
\end{align*}
\end{defn}

\begin{prop}
\label{prop:functional_dampening}
Consider an unbounded $A$-$B$-cycle $(E,\D)$ and a dampening function $f$. 
Assume that there exists a self-adjoint subset $W\subset\Lip^0(\D)\cap\Lip(f(\D))$ such that $\pi(A)\subset\overline{W}$. 
Then $(E,f(\D))$ is an unbounded $A$-$B$-cycle which is operator-homotopic to $(E,\D)$. 
\end{prop}
\begin{proof}
By \cref{lem:odd_funct_calc}, $f(\D)$ is an odd regular self-adjoint operator on $E$. 
Since the function $x\mapsto(1+f(x)^2)^{-\frac12}$ lies in $C_0(\R)$, we find that \[\Lip^0(\D)\cap\Lip(f(\D)) \subset \Lip^0(f(\D)).\] 
Hence $(E,f(\D))$ is indeed an unbounded $A$-$B$-cycle. 
To see that it is operator-homotopic to $(E,\D)$, consider the functions $g(x) := (1+x^2)^{-\frac12} \big(1+|f(x)|\big)$ and $h(x) := x g(x)$. 
Then $g\in C_{0}(\mathbb{R})$ and since $f-h\in C_b(\R)$, we see that $h(\D)$ is a bounded perturbation of $f(\D)$ (in particular, $(E,h(\D))$ is operator-homotopic to $(E,f(\D))$). 

It remains to show that $(E,h(\D))$ is operator-homotopic to $(E,\D)$. 
We consider the operator-homotopy given for $t\in[0,1]$ by 
\begin{align*}
\D_t &:= \D g_t(\D) , & 
g_t(x) &:= \big( (1-t)^{\frac12} + g(x) \big)^t . 
\end{align*}
We note that $g_0(x) = 1$ and $g_1(x) = g(x)$. 
Since $g(x)$ is bounded from below by a positive constant for $|x|<r$, we see that the map $[0,1]\ni t \mapsto g_t(\cdot) \in C_b(\R)$ is uniformly continuous on compact subsets of $\R$, and therefore $t \mapsto g_t(\D)$ is strongly continuous (see e.g.\ \cite[Lemma 7.2]{KL12}). Consequently, $t \mapsto \D_t$ is strongly continuous on $\Dom\D$. 
Furthermore, for each $t\in[0,1]$, $\Dom\D$ is a core for $\D_t$, so from \cref{lem:fam_reg_sa} we obtain a regular self-adjoint operator $\til\D$ on $C([0,1],E)$. 

Consider a self-adjoint element $w\in W$. 
Let us fix $0<t<1$ and write 
\[Q_t(\D) := (1-t)^{\frac12} + g(\D),\] 
so that $g_t(\D) = Q_t(\D)^t$. We note that $Q_t(\D)\in\Lip(\D)$ and $[\D,Q_t(\D)]=0$, and we find that 
\[
\D [Q_t(\D),w] = \D[g(\D),w] = [h(\D),w] - [\D,w]g(\D)
\]
is bounded. 
Consider the integral formula (see the proof of \cite[Proposition 1.3.8]{Pedersen79})
\begin{align}
\label{eq:integral_formula_powers}
Q_t(\D)^t = \frac{\sin(\pi t)}{\pi} \int_0^\infty \lambda^{-t} (1+\lambda Q_t(\D))^{-1} Q_t(\D) d\lambda . 
\end{align}
Since $Q_t(\D)$ is bounded below by $(1-t)^{\frac12}$, we know that $Q_t(\D)$ is invertible, and that 
\begin{align}
\label{eq:norm_bound_res_Q}
\|(1+\lambda Q_t(\D))^{-1}\| \leq (1-t)^{-\frac12} \lambda^{-1} .
\end{align}
In particular, $(1+\lambda Q_t(\D))^{-1}$ is of order $\order(\lambda^{-1})$ as $\lambda\to\infty$. 
Using that $\Dom \D$ is a core for $Q_{t}(D)$ and $\D$ commutes with $Q_t(\D)$, we then compute 
\begin{multline*}
\left[ (1+\lambda Q_t(\D))^{-1} Q_t(\D) , w \right] \D \\
= (1+\lambda Q_t(\D))^{-1} [Q_t(\D),w] \D - \lambda (1+\lambda Q_t(\D))^{-1} [Q_t(\D),w] \D (1+\lambda Q_t(\D))^{-1} Q_t(\D) ,
\end{multline*}
and we see that $\big\| \left[ (1+\lambda Q_t(\D))^{-1} Q_t(\D) , w \right] \D \big\|$ is finite and of order $\order(\lambda^{0})$ for $\lambda\to 0$, and of order $\order(\lambda^{-1})$ as $\lambda\to\infty$. By applying the above integral formula, we obtain that 
\[
S_t := [g_t(\D),w] \D = [Q_t(\D)^t,w] \D = \frac{\sin(\pi t)}{\pi} \int_0^\infty \lambda^{-t} \big[ (1+\lambda Q_t(\D))^{-1} Q_t(\D) , w \big] \D d\lambda 
\]
is a norm-convergent integral. It follows that $S_t$ is a bounded operator. To show that $S_t$ is in fact uniformly bounded in $t$, let us split the integral in two parts. First, since $\|(1+\lambda T)^{-1}\| \leq 1$, we have 
\begin{align*}
\Big\| &\frac{\sin(\pi t)}{\pi} \int_0^1 \lambda^{-t} \big[ (1+\lambda Q_t(\D))^{-1} Q_t(\D) , w \big] \D d\lambda \Big\| \\
&\leq \frac{\sin(\pi t)}{\pi} \big\| [g(\D),w]\D \big\| \big(1+\|Q_t(\D)\|\big) \int_0^1 \lambda^{-t} d\lambda \\
&\leq \frac{\sin(\pi t)}{\pi} \big\| [g(\D),w]\D \big\| \big(2+\|g(\D)\|\big) (1-t)^{-1} .
\end{align*}
Second, using \cref{eq:norm_bound_res_Q} we estimate
\begin{align*}
\Big\| &\frac{\sin(\pi t)}{\pi} \int_1^\infty \lambda^{-t} \big[ (1+\lambda Q_t(\D))^{-1} Q_t(\D) , w \big] \D d\lambda \Big\| \\
&\leq \frac{\sin(\pi t)}{\pi} \big\| [g(\D),w]\D \big\| \big( (1-t)^{-\frac12} + (1-t)^{-1} \|Q_t(\D)\| \big) \int_1^\infty \lambda^{-t-1} d\lambda \\
&\leq \frac{\sin(\pi t)}{\pi} \big\| [g(\D),w]\D \big\| \big( 2 (1-t)^{-\frac12} + (1-t)^{-1} \|g(\D)\| \big) t^{-1} .
\end{align*}
Using that $\sin(\pi t) = \order(t)$ as $t\to0$ and $\sin(\pi t) = \order(1-t)$ as $t\to1$, we see that both integrals are uniformly bounded in $t$. Thus $S_t$ is uniformly bounded. It then suffices to check strict continuity on the dense submodule $\Dom\D$. Since $g_t(\D)$ is strongly continuous, we see that $S_t$ is strongly continuous on $\Dom\D$. Furthermore, rewriting 
\begin{align*}
\D [g_t(\D),w] &= \left[ \D g_t(\D) , w \right] - \left[ \D , w \right] g_t(\D) 
= \left[ g_t(\D) , w \right] \D + g_t(\D) \left[ \D , w \right] - \left[ \D , w \right] g_t(\D) , 
\end{align*}
we conclude that $S_t^* = - \D [g_t(\D),w]$ is also strongly continuous on $\Dom\D$. 
Thus we have shown that the commutator 
\[
[\D_t,w] = [\D,w] g_t(\D) + \D [ g_t(\D) , w ]
\]
is uniformly bounded and strictly continuous, and therefore $[\til\D,w]$ is bounded and adjointable on $C([0,1],E)$. 

Now consider the functions $R_t\in C_0(\R)$ given by $R_t(x) := (i\pm xg_t(x))^{-1}$. We claim that $t\mapsto R_t$ is continuous with respect to the supremum-norm on $C_0(\R)$. 
To prove this claim, first observe that $g_t(x) \geq g(x)^t \geq \min(1,g(x))$ for all $x\in\R$ and $t\in[0,1]$. 
Hence for each $\varepsilon>0$ there exists $r\in(0,\infty)$ such that for all $t\in[0,1]$ we have $\sup_{|x|>r} |R_t(x)| \leq \varepsilon$. 
Then for $t,s\in[0,1]$ we can estimate 
\begin{align*}
\| R_t - R_s \| &\leq 2\varepsilon + \sup_{|x|<r} \|R_t(x)-R_s(x)\| 
\leq 2\varepsilon + \sup_{|x|<r} \|xg_t(x)-xg_s(x)\| \\
&\leq 2\varepsilon + r \sup_{|x|<r} \|g_t(x)-g_s(x)\| .
\end{align*}
Since $g_t(x)$ is uniformly continuous for $|x|<r$, we see that $t\mapsto R_t$ is norm-continuous. Consequently, we conclude that $t\mapsto (i\pm \D_t)^{-1}$ is a norm-continuous map such that $w (i\pm \D_t)^{-1}$ is compact for each $w\in W$ and $t\in[0,1]$. 
Hence $w (\til\D\pm i)^{-1}$ is compact on $C([0,1],E)$. 
This completes the proof that $\D_t$ yields an operator-homotopy $\big(C([0,1],E),\til\D\big)$. 
\end{proof}
\begin{remark}
\label{remark:higher-order}
A \emph{higher order} unbounded Kasparov module is a pair $(E,\D)$ such that there exist $0<\varepsilon<1$ and a dense $*$-subalgebra $\A\subset A$ for which the operators $[\D,a](1+\D^{2})^{-\frac{1-\varepsilon}{2}}$ (for $a\in\A$) extend to bounded operators. The class of higher order Kasparov modules contains all \emph{ordinary} unbounded Kasparov modules. 
In \cite[Theorem 1.37]{GMR19pre} it was shown that the function $$\sgnlog(x) := \sgn(x) \log(1+|x|),$$ can be used to turn a higher order unbounded Kasparov module into an ordinary unbounded Kasparov module. 
In fact, the proof of \cite[Theorem 1.37]{GMR19pre} shows that for any unbounded cycle $(E,\D)$ (as in \cref{def: Kasmod}) we have the inclusion $\Lip(\D) \subset \Lip(\sgnlog(\D))$. It then follows from \cref{prop:functional_dampening} that any unbounded cycle $(E,\D)$ is operator-homotopic to $(E,\sgnlog(\D))$. 

Using the natural notion of homotopy for higher order modules, one can ask whether the transformation $(E,\D)\mapsto(E,\sgnlog(\D))$ can be implemented as an operator-homotopy within the class of higher order unbounded Kasparov modules, so that every higher order module would be operator-homotopic to an ordinary unbounded Kasparov module. 
It is not immediately clear if this is indeed the case. 
\end{remark}
\subsection{From bounded to unbounded homotopies}
\label{sec:lift_homotopy}

Recall the $*$-homomorphism $\ev_t\colon C([0,1],B) \to B$ given by $b\mapsto b(t)$. 
For a Hilbert $C([0,1],B)$-module $E$ we write $E_{t}:=E\hot_{\ev_{t}}B$ for the localisation of $E$ at $t\in[0,1]$.
Moreover, for any $h\in\End_{B}^*(E)$, we consider the localisation $h_t := h\hot1$ on $E_t$. 
We describe some basic facts regarding these localisations in the Appendix. 

Now consider two unbounded  $A$-$B$-cycles $(E_0,\D_0)$ and $(E_1,\D_1)$, and assume that their bounded transforms are homotopic. 
Thus there exists a homotopy $( E,F)$ between $(E_0,F_{\D_0})$ and $(E_1,F_{\D_1})$, where $E$ is a module over $C([0,1],B)$. 
For simplicity, let us assume that $\ev_j(E,F)$ is \emph{equal} to $(E,F_{\D_j})$ (i.e.\ there is no unitary equivalence involved). We are ready to derive our main technical result.

\begin{prop}
\label{prop:unbdd_lift}
Suppose $A$ is separable, and $B$ $\sigma$-unital. 
Consider two unbounded $A$-$B$-cycles $(E_0,\D_0)$ and $(E_1,\D_1)$, and let $(E,F)$ be a homotopy between $(E_0,F_{\D_0})$ and $(E_1,F_{\D_1})$, satisfying $F=F^{*}$ and $F^{2}\leq 1$.
Let $W_j\subset\Lip^0(\D_j)$ be countable subsets consisting of products of elements in $\Lip^0(\D_j)$, such that $A\subset \overline{W_{j}}$ (for $j=0,1$). 
Then there exists a positive operator $ l\in J_{F}\subset \End_{C([0,1],B)}^*(E)$ with dense range in $E$ such that
\begin{enumerate}
\item the (closure of the) operator $\mathcal{S} := \frac12(F l^{-1} + l^{-1}  F)$ makes $(E,\mathcal{S})$ into an unbounded $A$-$C([0,1],B)$-cycle; 
\item writing $l_{j} := \ev_j( l) $ and $\mathcal{S}_{j}:=\ev_{j}(\mathcal{S})$ (for $j=0,1$), we have 
\begin{align*}
l_{j} &\in C^{*}((1+\D_j^{2})^{-1}) , &
\mS_j &= F_{\D_j} l_j^{-1} , &
W_j &\subset \Lip(l_j^{-1}) \cap \Lip(\mathcal{S}_{j}) , 
\end{align*} 
and the operator $l_j^{-1} (1+\D_j^2)^{-\frac14}$ extends to an adjointable endomorphism. 
\end{enumerate}
\end{prop}
\begin{proof}
Note that (1) can be obtained by an application of \cref{prop:lift}. In order to achieve (2) simultaneously, we need to construct our lift more carefully.
Consider again the $\sigma$-unital $C^*$-algebra $J_F = \End_{C([0,1],B)}^0(E) + C_F$. 
Let $k\in\End_{C([0,1],B)}^{0}(E)$ be an even strictly positive element and $\chi\in C([0,1])$ be given by $\chi(t) := t(1-t)$. Then $\chi k \in \End_{C([0,1],B)}^{0}(E)$ (cf.\ \cref{lem:localisations}), and we define
\[
 h := \chi k + (1-F^{2}) \in J_{F} .
\]
Consider the localisations 
$$h_{t} := \ev_t(h)=\chi(t)k_{t}+(1-F^{2}_{t}).$$
For $t\in (0,1)$ we have that $\chi(t)>0$, and $\chi(t)k_{t}$ has dense range in $E_{t}$ by \cref{locallydenserange}. Since  $\chi(t)k_t\leq h_t$, $h_{t}$ has dense range in $E_t$ by \cite[Corollary 10.2]{Lance95}. For $t\in\{0,1\}$, we have $h_{t}=(1-F^{2}_{t})^{\frac{1}{2}}=(1+\D^{2}_{t})^{-\frac{1}{2}}$, which has dense range as well. Thus, applying \cref{locallydenserange} again, we conclude that $h$ has dense range in $E$.
Moreover, from \cref{lem:J_F} it follows that $h$ is a strictly positive element in $J_F$. 

Let $\A:= \{a_{i}\}_{i\in\mathbb{N}}\subset A$ be a countable dense subset of $A$, let $\{c_{i}\}_{i\in \mathbb{N}}$ be a countable dense subset of $C^{*}(h)$, and let $\{w_{j,i}\}_{i\in\mathbb{N}}$ be an enumeration of $W_{j}$.  
We have the inclusions $A J_F, J_F A, F J_F, J_F F \subset J_F$ (see \cref{lem:J_F}). 
Since $\ev_j(F)=F_{\D_{j}}$ and $W_{j}\subset\Lip^{0}(\D_{j})$, we have for all $w\in W_{j}$ that $w(1-F^{2}_{{\D}_{j}})=w(1+\D_{j}^{2})^{-1}\in \End_B^0(E_j)$. 
Moreover, by assumption any $w\in W_j$ is of the form $w=T_1T_2$ for $T_1,T_2\in\Lip^0(\D_j)$. Since $[F_{\D_j},T_1]T_2$ is compact, as explained in the proof of \cref{prop:bdd_transform} (cf.\ \cite[Proposition 17.11.3]{Blackadar98}), it follows that also $[F_{\D_j},w] \in \End_B^0(E_j)$. 
It thus holds that 
$$W_{j}J_{F_{\D_{j}}},\, J_{F_{\D_{j}}}W_{j},\, F_{\D_{j}} J_{F_{\D_{j}}},\, J_{F_{\D_{j}}}F_{\D_{j}} \subset J_{F_{\D_{j}}}.$$ 
Furthermore, since $\ev_{j}\colon C([0,1],B)\to B=\End^{0}(B)$ is a surjective $*$-homomorphism we have $\End^{0}_B(E_{j})=\End^{0}_{C([0,1],B)}(E)\hot_{\ev_{j}}1$ and hence $J_{F_{\D_{j}}}=J_{F}\hot_{\ev_{j}}1$. Therefore any approximate unit $u_n\in J_{F}$ gives an approximate unit $\ev_{j}(u_{n})$ for $J_{F_{\D_{j}}}$. 
The $C^*$-subalgebra $C^*(h) \subset J_F$ thus contains a commutative approximate unit $u_{n}$ for $J_F$ which is quasicentral for $A$ and $F$, and such that for $j\in\{0,1\}$, $\ev_{j}(u_n)$ is quasicentral for $W_{j}$ \cite[Theorem 3.2]{AkPed}. 

By fixing a choice of $0<\varepsilon<1$ and selecting a suitable subsequence of $u_{n}$, we can  achieve that, for each $n\in\N$, $d_{n}:=u_{n+1}-u_{n}$ satisfies properties (a)-(e) of the proof of \cref{prop:lift} as well as

\begin{itemize}
\item[(c')] $\|\ev_j(d_{n})[\ev_j(F),w_{j,i}]\|\leq \varepsilon^{2n}$ for all $i\leq n$ and for $j=0,1$;
\item[(d')] $\|[\ev_j(d_{n}),w_{j,i}]\|\leq \varepsilon^{2n}$ for all $i\leq n$ and for $j=0,1$.
\end{itemize}

As in \cref{prop:lift}, property $(c')$ follows because $\ev_{j}(u_n)$ is an approximate unit for $J_{F_{\D_{j}}}$ and $(d')$ follows because $\ev_{j}(u_{n})$ is quasicentral for $W_{j}$. 
Thus, as in \cref{prop:lift}, we can construct a strictly positive element $l\in J_F$, such that the (closure of the) operator 
$$\mathcal{S} := \frac12(F  l^{-1} +  l^{-1} F)$$ is a densely defined and regular self-adjoint operator on $E$, and $(E,\mathcal{S})$ is an unbounded Kasparov $A$-$C([0,1],B)$-module for which we have $\A \subset \Lip^0(\mS)$. 
This proves (1). 

For (2), we first note that $l_j \in C^*(h_j)$ and $h_j = (1+\D_j^2)^{-1}$ for $j=0,1$. 
In particular, $l_j$ commutes with $F_{\D_j}$ and $\mS_j = F_{\D_j} l_j^{-1}$. 
Properties (c') and (d') ensure that $[\mathcal{S}_j,w]$ and $[l_j^{-1},w]$ are bounded for all $w\in W_j$ ($j=0,1$). 
Furthermore, from property (b) it follows that $ l^{-1}(1- F^{2})^{\frac{1}{4}}$ is everywhere defined and bounded, and localising in $j=0,1$ then shows that $l_j^{-1} (1+\D_j^2)^{-\frac14}$ is bounded.
\end{proof}

\begin{thm}
\label{thm:lift_homotopy}
Suppose $A$ is separable, and $B$ $\sigma$-unital. 
Consider two unbounded  $A$-$B$-cycles $(\pi_0,E_0,\D_0)$ and $(\pi_1,E_1,\D_1)$. 
Any homotopy $(\pi,E,F)$ between $(\pi_0,E_0,F_{\D_0})$ and $(\pi_1,E_1,F_{\D_1})$ can be lifted to an unbounded $A$-$C([0,1],B)$-cycle $(\pi,E,\mS)$ such that, for $j=0,1$, 
\begin{itemize}
\item the endpoints $\ev_j(\pi,E,\mS)$ are unitarily equivalent to $(\pi_j,E_j,f_j(\D_j))$ for dampening functions $f_j\colon\R\to\R$;
\item there exist countable self-adjoint subsets $W_j\subset\Lip^0(\D_j)\cap\Lip(f_j(\D_j))$ such that $\pi_j(A)\subset\overline{W_j}$. 
\end{itemize}
Moreover, if $(\pi,E,F)$ is an operator-homotopy, then $(\pi,E,\mS)$ is an operator-homotopy. 
\end{thm}
\begin{proof}
We may assume (without loss of generality) that $F=F^*$ and $F^2\leq1$ \cite[17.4.3]{Blackadar98}. 
For $j=0,1$, we have unitary equivalences $U_j \colon \ev_j( E) \to E_j$ such that $\ev_j( F) = U_j^* F_{\D_j} U_j$. Then $\D_j$ on $E_j$ is unitarily equivalent to $U_j^* \D_j U_j$ on $\ev_j( E)$. To simplify notation, we will from here on ignore this unitary equivalence and simply assume that $\ev_j( E, F)$ is \emph{equal} to $(E_j,F_{\D_j})$. 

We know by \cref{lem:separable_Lip} that, for $j=0,1$, there exist countable self-adjoint subsets $W_j \subset \Lip^0(\D_j)$ consisting of products of elements in $\Lip^0(\D_j)$, such that $\pi_j(A) \subset \overline{W_j}$. 
From \cref{prop:unbdd_lift}, we obtain an unbounded $A$-$C([0,1],B)$-cycle $\big(  E,\mathcal{S} := \frac12( F  l^{-1} +  l^{-1}  F) \big)$, which provides a homotopy between $(E_0,\mathcal{S}_0)$ and $(E_1,\mathcal{S}_1)$, where $\mathcal{S}_j := \ev_j(\mathcal{S})$. 
By property (2) of \cref{prop:unbdd_lift}, we know that $l_{j} \in C^{*}((1+\D_j^{2})^{-1})$, $\mathcal{S}_{j} = F_{\D_j} l_j^{-1}$, $W_j \subset \Lip(l_j^{-1}) \cap \Lip(\mathcal{S}_{j})$, and $l_j^{-1} (1+\D_j^2)^{-\frac14}$ is bounded. 
It follows that we can write $\mathcal{S}_j = f_j(\D_j)$ for some dampening function $f_j$, which proves the first statement. 
Furthermore, if we have in fact an \emph{operator}-homotopy $(E,F)$, then it is clear that the lift $( E,\mathcal{S})$ obtained from \cref{prop:unbdd_lift} is also an \emph{operator}-homotopy. 
\end{proof}

\subsection{The isomorphism with \texorpdfstring{$\KK$}{KK}-theory}
\label{sec:KK_isomorphism}

Using the results from the previous sections, we can now prove that our semigroup $\overline{\UKK}(A,B)$ is isomorphic to Kasparov's $\KK$-group.
\begin{thm}
\label{thm:bdd_transform_KK_isom}
Suppose $A$ is separable, and $B$ $\sigma$-unital. 
The bounded transform induces a semigroup isomorphism $\overline{\UKK}(A,B) \to \KK(A,B)$, given by $[(E,\D)] \mapsto [(E,F_\D)]$. 
\end{thm}
\begin{proof}
If there exists a homotopy $( E,\D)$ between unbounded  $A$-$B$-cycles $(E_0,\D_0)$ and $(E_1,\D_1)$, then $( E,F_{\D})$ provides a homotopy between $(E_0,F_{\D_0})$ and $(E_1,F_{\D_1})$. 
Moreover, the bounded transform is compatible with direct sums, so it induces a well-defined semi\-group homomorphism. 
Furthermore, this homomorphism is surjective by \cref{thm:bdd_transform_surj}, so it remains to prove that it is also injective. 

Consider two unbounded  $A$-$B$-cycles $(E_0,\D_0)$ and $(E_1,\D_1)$, such that $[(E_0,F_{\D_0})] = [(E_1,F_{\D_1})]$. Then there exists a homotopy $(E,F)$ between $(E_0,F_{\D_0})$ and $(E_1,F_{\D_1})$. 
From \cref{thm:lift_homotopy} we obtain an unbounded $A$-$C([0,1],B)$-cycle $(E,\mS)$ such that, for $j=0,1$, the endpoints $\ev_j(E,\mS)$ are unitarily equivalent to $(E_j,f_j(\D_j))$ for dampening functions $f_j\colon\R\to\R$, and there exist self-adjoint subsets $W_j\subset\Lip^0(\D_j)\cap\Lip(f_j(\D_j))$ such that $\pi_j(A)\subset\overline{W_j}$. 
It then follows from \cref{prop:functional_dampening} that $\D_j$ is operator-homotopic to $\mathcal{S}_{j}$. 
Thus we have the composition of homotopies 
\begin{align*}
\D_0 &\sim_{oh} \mathcal{S}_{0} \sim_{h} \mathcal{S}_{1} \sim_{oh} \D_1 ,
\end{align*}
which proves that $[(E_0,\D_0)] = [(E_1,\D_1)]$. 
\end{proof}

\begin{remark}
\label{remark:KK_group}
\emph{A priori}, $\overline{\UKK}(A,B)$ is a semigroup, and the isomorphism $$\overline{\UKK}(A,B) \to \KK(A,B),$$ is an isomorphism of semigroups. Since $\KK(A,B)$ is a group, it of course follows that $\overline{\UKK}(A,B)$ is also a group. 
However, the isomorphism $\overline{\UKK}(A,B) \to \KK(A,B)$ requires the assumption that $A$ is separable. 
In \cref{thm:UKK_group} we will give a direct proof that $\overline{\UKK}(A,B)$ is a group, which avoids the bounded transform and therefore also works for non-separable ($\sigma$-unital) $C^*$-algebras. 
\end{remark}

For any dense $*$-subalgebra $\A\subset A$, we define $\Psi_1(\A,B)$ as the set of those $(\pi,E,\D) \in \overline{\Psi}_1(A,B)$ for which $\pi(\A) \subset \Lip^0(\D)$, and we define $\UKK(\A,B)$ as the homotopy equivalence classes of elements in $\Psi_1(\A,B)$ (where it is understood that the homotopies are given by elements in $\Psi_1(\A,C([0,1],B))$). 
The natural inclusion $\Psi_1(\A,B) \into \overline{\Psi}_1(A,B)$ induces a well-defined semigroup homomorphism $\UKK(\A,B) \to \overline{\UKK}(A,B)$. We say that $\A$ is \emph{countably generated} if $\A$ contains a countable subset that generates it as a $*$-algebra over $\mathbb{C}$. We emphasize that this does not involve taking closures of any kind.
While, as we explained in \cref{remark:subalgebra}, it is \emph{not necessary} to fix a countably generated dense $*$-subalgebra $\A\subset A$, we will show next that it is nevertheless \emph{possible} to define unbounded $\KK$-theory using any such fixed choice for $\A\subset A$. 

\begin{prop}
Suppose $A$ is separable, and $B$ $\sigma$-unital. 
For any countably generated dense $*$-subalgebra $\A\subset A$, the map $\UKK(\A,B) \to \overline{\UKK}(A,B)$ is an isomorphism. 
\end{prop}
\begin{proof}
We have the following commuting diagram. 
\begin{align*}
\xymatrix{
\UKK(\A,B) \ar[rr] \ar[dr] && \overline{\UKK}(A,B) \ar[dl] \\
& \KK(A,B) & 
}
\end{align*}
We know from \cref{thm:bdd_transform_KK_isom} that the map $\overline{\UKK}(A,B) \to \KK(A,B)$ is an isomorphism. Thus we need to show that also $\UKK(\A,B) \to \KK(A,B)$ is an isomorphism. 
The assumption that $A$ is separable ensures that the bounded transform $\UKK(\A,B) \to \KK(A,B)$ is surjective (cf.\ \cref{thm:bdd_transform_surj}). 
Moreover, the proofs of \cref{thm:lift_homotopy,thm:bdd_transform_KK_isom} with the special choice $W_j = \pi_j(\A)$ show that the bounded transform is also injective. 
\end{proof}

\section{Degenerate cycles}
\label{sec:degenerate}

In this section, we will consider two notions of degenerate cycles in unbounded $\KK$-theory, namely `algebraically degenerate' and `spectrally degenerate' cycles. 
Our aim is to prove the following:
\begin{itemize}
\item any degenerate cycle is \emph{null-homotopic}, i.e.\ homotopic to the zero cycle $(0,0)$; 
\item any homotopy can be implemented as an operator-homotopy modulo addition of degenerate cycles. 
\end{itemize}

\subsection{Algebraically degenerate cycles}
\label{sec:alg_deg}

\begin{defn}
An unbounded  $A$-$B$-cycles $(\pi,E,\D)$ is called \emph{algebraically degenerate} if $\pi=0$. 
\end{defn}

By considering the obvious homotopy $(C_0((0,1],E),\D)$, we easily obtain: 
\begin{lem}
\label{lem:alg_deg_null}
An algebraically degenerate unbounded $A$-$B$-cycle $(E,\D)$ is null-homotopic. 
\end{lem}

As an application of the above lemma, we will show that two unbounded cycles $(\pi,E,\D )$ and $(\pi,E,\D')$ are homotopic if the difference $\D-\D'$ is \emph{`locally bounded'}. 

\begin{prop}
Let $(\pi,E,\D )$ and $(\pi,E,\D')$ be unbounded  $A$-$B$-cycles. 
Suppose there exists a subset $W \subset \Lip^0(\D)\cap\Lip^0(\D')$ with $\pi(A) \subset \overline{W}$ such that for each $w\in W$, the operator $(\D -\D')w$ extends to a bounded operator. 
Then $(\pi,E,\D )$ and $(\pi,E,\D')$ are homotopic. 
\end{prop}
\begin{proof}
Consider the unbounded $A$-$C([0,1],B)$-cycle $(\pi , C([0,1],E\oplus E) , \D \oplus \D')$ with the representation given for $t\in[0,1]$ by $\pi_t(a) := (a\oplus a) P_t$ in terms of the norm-continuous family of projections 
\[
P_t := \mattwo{\cos^2(\frac{\pi t}2)}{\cos(\frac{\pi t}2)\sin(\frac{\pi t}2)}{\cos(\frac{\pi t}2)\sin(\frac{\pi t}2)}{\sin^2(\frac{\pi t}2)} .
\]
We note that $P_0 = 1\oplus0$ and $P_1 = 0\oplus1$. 
For homogeneous $w\in W$ we compute 
\begin{multline*}
[ \D \oplus \D' , (w\oplus w) P_t ] = \\
\mattwo{[\D ,w]\cos^2(\frac{\pi t}2)}{(\D w-(-1)^{\deg w}w\D') \cos(\frac{\pi t}2) \sin(\frac{\pi t}2)}{(\D'w-(-1)^{\deg w}w\D ) \cos(\frac{\pi t}2)\sin(\frac{\pi t}2)}{[\D',w] \sin^2(\frac{\pi t}2)} .
\end{multline*}
We observe that $\D w-(-1)^{\deg w}w\D' = (\D -\D')w + [\D',w]$ is bounded, and similarly for $\D'w-(-1)^{\deg w}w\D $. 
Hence $[ \D \oplus \D' , (w\oplus w) P_t ]$ is uniformly bounded and norm-continuous in $t$, and we obtain $(w\oplus w)P_\bullet \subset \Lip(\D\oplus\D')$. 
Moreover, since the resolvents of $\D\oplus\D'$ are constant in $t$, we have in fact $(w\oplus w)P_\bullet \subset \Lip^0(\D\oplus\D')$. 
Thus we have 
\[
\pi_\bullet(A) \subset \overline{\{ (w\oplus w)P_\bullet \mid w\in W \}} \subset \overline{\Lip^0(\D\oplus\D')} , 
\]
and we have a homotopy between $(\pi\oplus0 , E\oplus E , \D \oplus \D')$ and $( 0\oplus\pi , E\oplus E , \D \oplus \D')$. Finally, since algebraically degenerate cycles are null-homotopic by \cref{lem:alg_deg_null}, we note that $(\pi\oplus0 , E\oplus E , \D \oplus \D')$ is homotopic to $(\pi,E,\D )$, and that $( 0\oplus\pi , E\oplus E , \D \oplus \D')$ is homotopic to $(\pi,E,\D')$. 
\end{proof}

\begin{remark}
The assumption that $(\D-\D')w$ is bounded for all $w\in W$ is interpreted as saying that $\D-\D'$ is \emph{locally bounded}. 
In the above proposition, we have assumed that both $(E,\D)$ and $(E,\D')$ are unbounded cycles. Under certain conditions, it suffices to assume only that $(E,\D)$ is an unbounded cycle; using local boundedness of $\D-\D'$ one can then \emph{prove} that $(E,\D')$ is also an unbounded cycle. We refer to \cite{vdD18} for further details. 
\end{remark}

\subsection{Spectrally degenerate cycles}
\label{sec:spec_deg}

We denote by $\sgn:\mathbb{R}\setminus\{0\}\to \{\pm 1\}$ the function $\sgn (x):=\frac{x}{|x|}$. We say that a regular self-adjoint operator $\D:\Dom \D\to E$ is \emph{invertible} if there exists $\D^{-1}\in\End^{*}_{B}(E)$ that satisfies $\D\D^{-1}=\D^{-1}\D=1$. It then follows that $\Dom\D=\Ran\D^{-1}=\Ran |\D|^{-1}$ and $\Ran\D=E$.
Thus if $\D$ is invertible, $\sgn(\D)$ is well-defined and equal to $\D|\D|^{-1}$. 

\begin{defn}
An unbounded $A$-$B$-cycle $(\pi,E,\D)$ is called \emph{spectrally degenerate} if $\D$ is invertible and there exists $W\subset\Lip^{0}(\D)$ such that $\pi(A)\subset\overline{W}$ and $[\sgn(\D),w] = 0$ for all $w\in W$. 
\end{defn}

\begin{lem}
\label{lem:spec_deg_Lip_reg}
Let $\D:\Dom \D\to E$ be self-adjoint, regular, and invertible. If $w\in\End^{*}_{B}(E)$ is such that $w\colon\Dom\D\to\Dom\D$ and $[\sgn(\D),w]=0$, 
then $[\D,w]$ is bounded if and only if $[|\D|,w]$ is bounded. 
\end{lem}
\begin{proof}
This follows from the simple observation that $\sgn(\D)$ is a self-adjoint unitary and $\D=\sgn(\D)|D|$. We have
\[[\D,w]=\sgn(\D)[|\D|,w],\quad [|\D|,w]=\sgn(\D)[\D,w],\]
whence $[\D,w]$ is bounded if and only if $[|\D|,w]$ is bounded.
\end{proof}

We have already seen in \cref{lem:alg_deg_null} that any algebraically degenerate cycle is null-homotopic. 
Here, we shall prove that also any spectrally degenerate cycle $(E,\D)$ is null-homotopic. 
The easiest way to prove this is by observing that the bounded transform $(E,F_\D)$ is operator-homotopic to the degenerate cycle $(E,\sgn(\D))$ (which is null-homotopic), and then applying \cref{thm:lift_homotopy}. 
However, we can only apply \cref{thm:lift_homotopy} if $A$ is separable. But with only a bit more effort, we can in fact explicitly construct an \emph{unbounded} homotopy between any spectrally degenerate cycle and the zero module. 

\begin{prop}
\label{prop:deg_hom_0}
Any spectrally degenerate unbounded $A$-$B$-cycle $(E,\D)$ is null- \\ homotopic. 
\end{prop}
\begin{proof}
Consider for $t\in(0,1]$ the family of regular self-adjoint operators 
\[
\D_t := t^{-1} \sgn(\D) |\D|^t . 
\]
Since $t \mapsto |\D|^{t-1}$ is norm-continuous and $|\D|^t = |\D|^{t-1} |\D|$, we see that $|\D|^t$ is strongly continuous on $\Dom\D$. 
Since $\Dom\D$ is a core for $\D_t$ for each $t\in(0,1]$, we obtain from \cref{lem:fam_reg_sa} a regular self-adjoint operator $\til\D$ on the Hilbert $C([0,1],B)$-module $\til E := C_0((0,1],E)$. 
We claim that $(\til E,\til\D)$ is an unbounded cycle, and therefore it provides a homotopy between $\ev_1(\til E,\til\D) = (E,\D)$ and $\ev_0(\til E,\til\D) = (0,0)$. 

To prove the claim, choose $W\subset\Lip^0(\D)$ such that $\pi(A)\subset\overline{W}$ and  $[\sgn(\D),w] = 0$ for all $w\in W$. 
First consider the resolvents of $\D_t$. We compute
\begin{equation}
\label{eq:res_inv}
(\D_t\pm i)^{-1}=\mp it \sgn(\D)|\D|^{-t} \big(t\sgn(\D)|\D|^{-t} \mp i\big)^{-1} 
\end{equation}
Since $\D$ is invertible, the operators $w|\D|^{-t}$ are compact for $0<t\leq 1$ and for $w\in W$, and hence so are $w(\D_{t}\pm i)^{-1}$. 
Moreover, $t\mapsto |\D|^{-t}$ is norm-continuous on $(0,1]$, and therefore $t \mapsto t\sgn(\D)|\D|^{-t}$ is norm-continuous on $(0,1]$. 
But then the composition with $x\mapsto x(x\pm i)^{-1}$ gives again a continuous function, and we see from \cref{eq:res_inv} that $t \mapsto (\D_t\pm i)^{-1}$ is norm-continuous on $(0,1]$. 
Furthermore, since $|\D|^{-t}$ is uniformly bounded and $t\sgn(\D)|\D|^{-t}$ is self-adjoint, it also follows from \cref{eq:res_inv} that 
\[\lim_{t\searrow 0}\big\|(\D_t\pm i)^{-1}\big\|=\lim_{t\searrow 0}t \big\| \sgn(\D)|\D|^{-t} \big(t\sgn(\D)|\D|^{-t} \mp i\big)^{-1}\big\|=0,\]
so we also obtain continuity at $0$. Hence $w(\til\D\pm i)^{-1}$ is compact on $\til E$. 

Next, we consider the commutator $[\D_t,w] = t^{-1} \sgn(\D) [|\D|^t,w]$ for some self-adjoint $w \in W$. 
We have seen above that $|\D|^t$ is strongly continuous on $\Dom\D$, and hence also $[\D_t,w]$ is strongly continuous on $\Dom\D$. 
To show that $[\D_t,w]$ is strongly continuous everywhere, it then suffices to show that $[\D_t,w]$ is uniformly bounded. 
For this purpose, we consider the operator inequality
\[
\pm i \big[ |\D|^{-1},w \big] = \mp i |\D|^{-1} \big[ |\D|,w \big] |\D|^{-1} 
\leq \big\| \big[ |\D|,w \big] \big\| |\D|^{-2} ,
\]
where $[|\D|,w]$ is bounded by \cref{lem:spec_deg_Lip_reg}.
Applying \cite[Proposition 2.11]{Kuc00} to the function $f(x) := x^t$, we then find that
\begin{align*}
\pm i \big[ |\D|^{-t},w \big] 
&= \pm i \big[ f\big(|\D|^{-1}\big),w \big] 
\leq f'\big(|\D|^{-1}\big) \; \big\| \big[ |\D|,w \big] \big\| \; |\D|^{-2} \\
&= t |\D|^{1-t} \; \big\| \big[ |\D|,w \big] \big\| \; |\D|^{-2} 
= t \big\| \big[ |\D|,w \big] \big\| \; |\D|^{-1-t} . 
\end{align*}
For any $\psi\in\Dom\D$, we therefore have
\begin{align*}
\Big\la \psi \Bigmvert \pm i \big[ |\D|^{t},w \big] \psi \Big\ra 
&= \Big\la \psi \Bigmvert \mp i |\D|^t \big[ |\D|^{-t},w \big] |\D|^t \psi \Big\ra 
= \Big\la |\D|^t \psi \Bigmvert \mp i \big[ |\D|^{-t},w \big] |\D|^t \psi \Big\ra \\
&\leq \Big\la |\D|^t \psi \Bigmvert t \big\| \big[ |\D|,w \big] \big\| \; |\D|^{-1} \psi \Big\ra 
= \Big\la \psi \Bigmvert t \big\| \big[ |\D|,w \big] \big\| \; |\D|^{t-1} \psi \Big\ra .
\end{align*}
Since both $\|[|\D|,w]\| \, |\D|^{t-1} $ and $[|D|^{t},w]$ are bounded for $t\in [0,1]$ (for the latter, see for instance \cite[Lemma 10.17]{ElementsNCG}), we have the norm-inequality 
\[
\big\| \big[ |\D|^{t},w \big] \big\| =\big\| \pm i \big[ |\D|^{t},w \big] \big\| \leq t \big\| \big[ |\D|,w \big] \big\| \; \big\| |\D|^{t-1} \big\| \leq t \big\| \big[ |\D|,w \big] \big\| \; \max\big\{1,\big\| |\D|^{-1} \big\|\big\} .
\]
We finally obtain 
\begin{align*}
\big\| [\D_t,w] \big\| 
&\leq t^{-1} \big\| \sgn(\D) \big\| \; \big\| \big[ |\D|^{t},w \big] \big\| 
\leq \big\| \big[ |\D|,w \big] \big\| \; \max\big\{1,\big\| |\D|^{-1} \big\|\big\} .
\end{align*}
Hence $[\D_t,w]$ is uniformly bounded and strongly continuous as a function of $t\in(0,1]$, and therefore the commutator $[\til\D,w]$ is bounded on $\til E$. 
Thus we have shown that $W \subset \Lip^0(\til\D)$ and therefore $\til\pi(A) \subset \overline{W} \subset \overline{\Lip^0(\til\D)}$. 
\end{proof}

\subsection{Operator-homotopies modulo degenerate cycles}
\label{sec:op-hom_mod_null}

In bounded $\KK$-theory, it was shown by Kasparov that any homotopy can be implemented as an operator-homotopy modulo addition of degenerate modules \cite[\S6, Theorem 1]{Kas80}. 
Using this result, we will prove that a similar statement holds in unbounded $\KK$-theory. 

Let $\sim_{oh+d}$ denote the equivalence relation on $\overline{\Psi}_1(A,B)$ given by operator-homotopies, unitary equivalences, and addition of spectrally degenerate and algebraically degenerate cycles. 
We already know from \cref{lem:alg_deg_null,prop:deg_hom_0} that degenerate cycles are null-homotopic, so $\sim_{oh+d}$ is stronger than $\sim_h$. 
We will prove here that in fact these two relations coincide. 

\begin{lem}
\label{lem:deg_lift}
Suppose $A$ is separable, and $B$ $\sigma$-unital. 
Let $(E,F)$ be a (bounded) Kasparov $A$-$B$-module, such that $F=F^*$, $F^2=1$, and $[F,a]=0$ for all $a\in A$ (in particular, $(E,F)$ is degenerate).
Let $\D := Fl^{-1}$ be a lift of $F$, where $l$ is a positive element in $J_F$ with dense range in $E$ obtained from \cref{prop:lift}. 
Then the unbounded $A$-$B$-cycle $(E,\D)$ is spectrally degenerate. 
\end{lem}
\begin{proof}
Since $F^2=1$ and $[F,l]=0$, we have that $\D$ is invertible and that $\sgn(\D) = F$ (graded) commutes with the algebra $A$. Thus $(E,\D)$ is spectrally degenerate. 
\end{proof}

\begin{thm}
\label{thm:lift_op-hom_mod_null}
Suppose $A$ is separable, and $B$ $\sigma$-unital. 
Then the homotopy equivalence relation $\sim_h$ on $\overline{\Psi}_1(A,B)$ coincides with the equivalence relation $\sim_{oh+d}$. 
\end{thm}
\begin{proof}
We need to prove that the relation $\sim_{oh+d}$ is weaker than $\sim_h$. 
To this end let $(E_0,\D_0)$ and $(E_1,\D_1)$ be unbounded $A$-$B$-cycles which are homotopic. 
We then know that the bounded transforms $(\pi_0,E_0,F_{\D_0})$ and $(\pi_1,E_1,F_{\D_1})$ are also homotopic. 
By \cite[\S6, Theorem 1]{Kas80}, there exist degenerate bounded Kasparov modules $(\pi_0',E_0',F_0')$ and $(\pi_1',E_1',F_1')$ such that $(\pi_0\oplus\pi_0',E_0\oplus E_0',F_{\D_0}\oplus F_0')$ is operator-homotopic to $(\pi_1\oplus\pi_1',E_1\oplus E_1',F_{\D_1}\oplus F_1')$. 
Denote by $E_j'^\op$ the Hilbert $B$-module $E_j'$ equipped with the opposite $\Z_2$-grading. 
By adding the algebraically degenerate module $(0,E_0'^\op,-F_0') \oplus (0,E_1'^\op,-F_1')$, we obtain the top line in the following diagram: 
\begin{align}
\label{eq:bdd_op-hom}
\xymatrix{
F_{\D_0} \oplus F_0' \oplus -F_0' \oplus -F_1' \ar@{~}[r]^-{oh} \ar@{~}[d]^-{oh} & F_{\D_1} \oplus F_1' \oplus -F_0' \oplus -F_1' \ar@{~}[d]^-{oh} \\
F_{\D_0} \oplus \hat{F_0'} \oplus -F_1' \ar@{~}[d]^-{oh} & F_{\D_1} \oplus \hat{F_1'} \oplus -F_0' \ar@{~}[d]^-{oh} \\
F_{\D_0} \oplus F_{\hat\D'_{0}} \oplus -F_{\D_1'} \ar@{~}[r]^-{oh} & F_{\D_1} \oplus F_{\hat\D'_{1}} \oplus -F_{\D_0'} \\
}
\end{align}
Since $F_j'$ is degenerate, we know for all $a\in A$ that $[F_j',\pi_j'(a)] = \pi_j'(a) (1-(F_j')^2) = 0$. 
As in \cite[\S17.6]{Blackadar98}, $\big(\pi_j'\oplus0,E_j'\oplus E_j'^\op,F_j'\oplus-F_j'\big)$ is operator-homotopic to the degenerate module 
\begin{align*}
&\big( \pi_j'\oplus0 , E_j'\oplus E_j'^\op , \hat{F_j'} \big) , & 
\hat{F_j'} &:= \mattwo{F_j'}{(1-(F_j')^2)^{\frac12}}{(1-(F_j')^2)^{\frac12}}{-F_j'} . 
\end{align*}
This yields the vertical operator-homotopies between the first two lines in \cref{eq:bdd_op-hom}. 

By construction, $(\hat{F_j'})^2 = 1$ and $[\hat{F_j'},\pi_j'(a)] = 0$. 
Hence by \cref{lem:deg_lift,prop:lift} we can lift $\hat{F_j'}$ to spectrally degenerate unbounded cycles $(\pi_j'\oplus0,E'_{j}\oplus E_j'^\op,\hat\D'_{j})$, such that $\hat{F_j'} \sim_{oh} F_{\hat\D'_{j}}$. 
Moreover, using again \cref{prop:lift}, we can lift $-F_j'$ to algebraically degenerate unbounded cycles $(0,E_j'^\op,-\D_j')$ such that $-F_j' \sim_{oh} -F_{\D_j'}$. 
This yields the vertical operator-homotopies between the last two lines in \cref{eq:bdd_op-hom}. 
Finally, by transitivity we obtain the horizontal operator-homotopy on the bottom line, and by \cref{thm:lift_homotopy} this operator-homotopy lifts to an unbounded operator-homotopy
\begin{align*}
&\xymatrix{
\D_0 \oplus \hat\D'_0 \oplus -\D'_1 \ar@{~}[r]^-{oh} & \D_1 \oplus \hat\D'_1 \oplus -\D'_0 . 
}
\end{align*}
Thus we have shown that $(E_0,\D_0) \sim_{oh+d} (E_1,\D_1)$. 
\end{proof}

\section{Symmetries and the group structure}
\label{sec:symmetries}

In this section we discuss various notions of symmetries for unbounded cycles. The presence of such symmetries induces homotopical triviality and can be used to give a direct proof of the fact that the semigroup $\overline{\UKK}(A,B)$ is a group for any two $\sigma$-unital $C^{*}$-algebras.

\subsection{Lipschitz regularity}
\label{sec:Lip_reg}

Let $0<\alpha<1$ and $f_\alpha\in C_0(\R)$ a function that behaves like $x^\alpha$ towards infinity. 
We will show here that we can use the functional dampening of \cref{prop:functional_dampening} to replace any unbounded cycle $(E,\D)$ by a \emph{Lipschitz regular} cycle $(E,f_\alpha(\D))$. 

\begin{defn}
\label{defn:Lip_reg}
An unbounded $A$-$B$-cycle $(\pi,E,\D)$ is called \emph{Lipschitz regular} if $\pi(A) \subset \overline{\Lip^0(\D)\cap\Lip(|\D|)}$. 
\end{defn}

\begin{remark}
Since the map $x \mapsto |x| - (1+x^2)^{\frac12}$ lies in $C_0(\R)$, we have for $T\in\End_B^*(E)$ that $[|\D|,T]$ is bounded if and only if $[(1+\D^2)^{\frac12},T]$ is bounded, and therefore $\Lip(|\D|) = \Lip((1+\D^2)^{\frac12})$. 
\end{remark}

The following result generalises \cite[Proposition 5.1]{Kaa19pre}, where the specific function $x \mapsto x(1+x^2)^{\frac{\alpha-1}2}$ was considered. 
\begin{prop}
\label{prop:power-damp_Lip_reg}
Let $(E,\D)$ be an unbounded cycle, $0<\alpha<1$, and let $f_\alpha\colon\R\to\R$ be any odd continuous function such that $\lim_{x\to\infty} f_\alpha(x) - x^\alpha$ exists. 
Then $(E,f_\alpha(\D))$ defines a Lipschitz regular unbounded cycle that is operator homotopic to $(E,\D)$.
\end{prop}
\begin{proof}
We will show that $\Lip(\D) \subset \Lip(f_\alpha(\D)) \cap \Lip(|f_\alpha(\D)|)$, and the statement then follows from \cref{prop:functional_dampening}. 
Given two such functions $f_\alpha$ and $g_\alpha$, both $f_\alpha-g_\alpha$ and $|f_\alpha|-|g_\alpha|$ lie in $C_b(\R)$. 
Thus $f_\alpha(\D) - g_\alpha(\D)$ and $|f_\alpha(\D)| - |g_\alpha(\D)|$ are bounded operators, and we see that $\Lip(f_\alpha(\D)) = \Lip(g_\alpha(\D))$ and $\Lip(|f_\alpha(\D)|) = \Lip(|g_\alpha(\D)|)$. 
Hence it suffices to prove the statement for $f_\alpha(x) := x (1+x^{2})^{\frac{\alpha-1}{2}}$. 
Using for $s\in(0,1)$ the integral formula (which can be derived from \cref{eq:integral_formula_powers} by taking $T = (1+\D^2)^{-1}$)
\[
(1+\D^2)^{-s} = \frac{\sin(\pi s)}{\pi} \int_0^\infty \lambda^{-s} (1+\lambda+\D^2)^{-1} d\lambda ,
\]
it is shown in the proof of \cite[Proposition 5.1]{Kaa19pre} that $[(1+\D^{2})^{\frac{\alpha-1}{2}},T]\D$ extends to a bounded operator for each $T\in\Lip(\D)$. 
Hence $\Lip(\D) \subset \Lip(f_\alpha(\D))$. 

To prove the Lipschitz regularity, we consider instead the function $g_\alpha(x) := \sgn(x) (1+x^2)^{\frac\alpha2}$. 
Using again the above integral formula, one can show similarly that 
\[
\big[ |g_\alpha(\D)| , T \big] = \big[ (1+\D^2)^{\frac\alpha2} , T \big] = - (1+\D^2)^{\frac\alpha2} \big[ (1+\D^2)^{-\frac\alpha2} , T \big] (1+\D^2)^{\frac\alpha2} 
\]
is indeed bounded for each $T\in\Lip(\D)$, and therefore $\Lip(\D) \subset \Lip(|g_\alpha(\D)|)$. 
\end{proof}

\begin{remark}
\label{remark:sgnmod}
In addition to the functions $x \mapsto x(1+x^2)^{\frac{\alpha-1}2}$ and $x \mapsto \sgn(x) (1+x^2)^{\frac\alpha2}$ considered in the proof of \cref{prop:power-damp_Lip_reg}, another typical example of a function $f_\alpha$ as in \cref{prop:power-damp_Lip_reg} is the function $\sgnmod^{\alpha}:\mathbb{R}\to \mathbb{R}$ given by $x\mapsto \sgn(x)|x|^{\alpha}$. 
Note that if $\D$ is invertible, then $\sgnmod^{\alpha}(\D)=\sgn(\D)|\D|^{\alpha} = \D |\D|^{\alpha-1}$.
\end{remark}

\begin{remark}
Recall from \cref{remark:higher-order} the function $\sgnlog(x) := \sgn(x) \log(1+|x|)$. 
In \cite[Theorem 1.16]{GMR19pre}, it is proved that the transformation $\D\mapsto\sgnlog(\D)$ turns Lipschitz regular \emph{twisted} unbounded Kasparov modules into ordinary unbounded Kasparov modules. Incorporating this `untwisting' procedure into the homotopy framework using \cref{prop:functional_dampening} is of interest in the study of twisted local index formulae. This is beyond the scope of the present paper.
\end{remark}

\subsection{Spectral symmetries}
\label{sec:spec_symm}

\begin{defn}
An unbounded $A$-$B$-cycle $(E,\D)$ is called 
\begin{itemize}
\item \emph{spectrally symmetric} if there exist an odd self-adjoint unitary $S$ on $E$ and a $W\subset\Lip^0(\D)$ such that $\pi(A)\subset\overline{W}$,  $[S,w]=0$ for all $w\in W$, $S\colon \Dom\D\to\Dom\D$, $\D S - S\D = 0$, and $S\D$ is positive; 
\item \emph{spectrally decomposable} if there exists a spectral symmetry $S$ such that both $(S\pm1)\D$ are positive. 
\end{itemize}
\end{defn}
The definition of spectrally decomposable cycle is adapted from \cite[Definition 4.1]{Kaa19pre} (where it is phrased in terms of the projection $P=\frac12(1+S)$). 
By definition, every spectrally decomposable cycle is also spectrally symmetric. 
Moreover, any spectrally degenerate cycle $(E,\D)$ is clearly spectrally decomposable (hence spectrally symmetric) with spectral symmetry $\sgn(\D)$. 

Spectrally symmetric cycles are actually not much more general than spectrally degenerate cycles. 
Indeed, the following lemma shows that any spectral symmetry $S$ more or less acts like $\sgn(\D)$ (except that $\D$ may not be invertible, so there could be some freedom in how $S$ acts on $\Ker\D$). 

\begin{lem}
\label{lem:spec_symm_almost_deg}
Let $(E,\D)$ be an unbounded $A$-$B$-cycle with spectral symmetry $S$. Then $\D = S|\D|$, and $(E,\D)$ is Lipschitz regular. 
\end{lem}
\begin{proof}
On the $\Z_2$-graded module $E = E_+\oplus E_-$ we can write
\begin{align*}
\D &= \mattwo{0}{\D_-}{\D_+}{0} , & 
S &= \mattwo{0}{U^*}{U}{0} , 
\end{align*}
where $U\colon E_+\to E_-$ is unitary. 
Since $\D S=S\D$, we see that $U \D_- = \D_+ U^*$. 
We then compute
\[
\D^2 = \mattwo{\D_-\D_+}{0}{0}{\D_+\D_-} = \mattwo{U^*\D_+U^*\D_+}{0}{0}{U\D_-U\D_-} . 
\]
Since $S\D$ is positive, we know that $U^*\D_+$ and $U\D_-$ are positive, and we obtain
\begin{align*}
|\D| &= \mattwo{U^*\D_+}{0}{0}{U\D_-} = S\D .
\end{align*}
As in \cref{lem:spec_deg_Lip_reg}, it then follows that $\Lip(|\D|) = \Lip(\D)$, so in particular $(E,\D)$ is Lipschitz regular. 
\end{proof}

Furthermore, the next proposition shows that any spectrally symmetric cycle is in fact just a bounded perturbation of a spectrally degenerate cycle. 

\begin{prop}
\label{prop:spec_symm_bdd_pert_deg}
Let $(E,\D)$ be an unbounded  $A$-$B$-cycle with spectral symmetry $S$. Then $(E,\D+S)$ is a spectrally degenerate unbounded $A$-$B$-cycle. 
\end{prop}
\begin{proof}
Since $S$ is bounded, self-adjoint, and odd, we know that $(E,\D+S)$ is again an unbounded $A$-$B$-cycle. 
Furthermore, since $(\D+S)^2 = \D^2 + 1 + 2S\D$ is positive and invertible, we know that also $\D+S$ is invertible. 
Moreover, noting that $(\D+S)^2 = (1+S\D)^2$ and that $1+S\D$ is positive, we see that $|\D+S| = 1+S\D$. 
Hence we find that 
\[
\sgn(\D+S) = (\D+S) |\D+S|^{-1} = S (S\D+1) (1+S\D)^{-1} = S , 
\]
and we conclude that $(E,\D+S)$ is degenerate. 
\end{proof}

In \cite[Definition 4.8]{Kaa19pre}, the notion of spectrally decomposable module was used to define the equivalence relation of `stable homotopy' for unbounded Kasparov modules (i.e.\ homotopies modulo addition of spectrally decomposable modules). Here, we point out that in fact any spectrally symmetric cycle $(E,\D)$ is null-homotopic. 
If $A$ is separable, this follows from \cref{thm:lift_homotopy} by observing that, if $S$ is a spectral symmetry of $(E,\D)$, then the bounded transform $(E,F_\D)$ is operator-homotopic to the degenerate cycle $(E,S)$ (since $[F_\D,S] = 2 S F_\D$ is positive, cf.\ \cite[Proposition 17.2.7]{Blackadar98}). 
In general, we simply combine \cref{prop:spec_symm_bdd_pert_deg,prop:deg_hom_0} to obtain: 

\begin{coro}
\label{coro:spec_symm_deg_0}
Any spectrally symmetric unbounded  $A$-$B$-cycle is null-homo\-topic. 
Consequently, the relation of stable homotopy equivalence of \cite[Definition 4.8]{Kaa19pre} coincides with the relation $\sim_{h}$ of homotopy equivalence.
\end{coro}

In \cite[Theorem 7.1]{Kaa19pre} it was shown that, for any countable dense $*$-subalgebra $\A\subset A$, the stable homotopy equivalence classes of elements in $\Psi_1(\A,B)$ form a group which is isomorphic to $\KK(A,B)$. In particular, this group is independent of the choice of $\A$. 
We emphasise here that \cref{coro:spec_symm_deg_0}, combined with \cite[Theorem 7.1]{Kaa19pre}, then gives a second independent proof of the isomorphism $\overline{\UKK}(A,B) \simeq \KK(A,B)$ from \cref{thm:bdd_transform_KK_isom}. 

As a further application of \cref{coro:spec_symm_deg_0}, the following proposition (adapted from the results of \cite{Kaa19pre}) gives a criterion that ensures that two given unbounded cycles are homotopic. 

\begin{prop}[cf.\ {\cite[Proposition 6.2]{Kaa19pre}}]
\label{prop:equal_spec_symm}
Let $(\pi,E,\D)$ and $(\pi,E,\D')$ be unbounded $A$-$B$-cycles such that $\pi(A) \subset \overline{\Lip^0(\D)\cap\Lip^0(\D')}$. Suppose there exists an odd self-adjoint unitary $F\colon E\to E$ such that $F$ commutes with both $\D$ and $\D'$, and such that we have the equalities $F\D = |\D|$ and $F\D' = |\D'|$. 
Then $(E,\D)$ is homotopic to $(E,\D')$. 
\end{prop}
\begin{proof}
Using \cref{prop:power-damp_Lip_reg}, we may assume (without loss of generality) that $(E,\D)$ and $(E,\D')$ are Lipschitz regular, and that $\pi(A) \subset \overline{W}$ for some $$W \subset \Lip^0(\D)\cap\Lip(|\D|)\cap\Lip^0(\D')\cap\Lip(|\D'|). $$
We then note that the operator $F$ satisfies the assumptions of \cite[Proposition 6.2]{Kaa19pre} (with the dense $*$-subalgebra $\A\subset A$ replaced by $W$), where we point out that the Lipschitz regularity of $\D$ ensures that 
\[
\D [F,w] = [\D F,w] - [\D,w] F = [|\D|,w] - [\D,w] F 
\]
is bounded for $w\in W$ (and similarly for $\D'$). 
Then we know from (the proof of) \cite[Proposition 6.2]{Kaa19pre} that $(E,\D) - (E,\D')$ is homotopic to a spectrally decomposable cycle. 
Using \cref{coro:spec_symm_deg_0} we conclude that $(E,\D)-(E,\D')$ is null-homotopic, and therefore $(E,\D)$ is homotopic to $(E,\D')$. 
\end{proof}

\begin{coro}
Let $(\pi,E,\D)$ and $(\pi,E,\D')$ be unbounded  $A$-$B$-cycles. Suppose that $\D$ and $\D'$ are both invertible, and that $\sgn(\D) = \sgn(\D')$. 
Then $(E,\D)$ is homotopic to $(E,\D')$. 
\end{coro}

\subsection{Clifford symmetries}
\label{sec:Cliff_symm}

\begin{defn}
An unbounded Kasparov $A$-$B$-cycle $(E,\D)$ is called \emph{Clifford symmetric} if there exist an odd self-adjoint unitary $\gamma$ on $E$ and a $W\subset\Lip^0(\D)$ such that $\pi(A)\subset\overline{W}$, $[\gamma,w]=0$ for all $w\in W$, $\gamma\colon \Dom\D\to\Dom\D$, and $\D \gamma=-\gamma\D$. 
\end{defn}
The idea here is that a Clifford symmetric $A$-$B$-cycle is in fact an $A\hot\CCliff_1$-$B$-cycle, and the image of the map $\KK(A\hot\CCliff_1,B)\to\KK(A,B)$ is zero. Indeed, one easily checks that the bounded transform $(E,F_\D)$ of a Clifford symmetric unbounded cycle is operator-homotopic to the degenerate Kasparov module $(E,\gamma)$. 
We prove here an analogous statement for unbounded cycles. 

\begin{lem}
\label{lem:Cliff_symm_spec_symm}
Let $(E,\D)$ be an unbounded $A$-$B$-cycle with a Clifford symmetry $\gamma$ and $0<\alpha <1$. 
Then $(E,\D)$ is operator-homotopic to the spectrally symmetric unbounded cycle $(E,\gamma|\D|^{\alpha})$. 
\end{lem}
\begin{proof}
Since $\gamma$ commutes with $|\D|^{\alpha}$ and $(\gamma|\D|^{\alpha})^2 = |\D|^{2\alpha}$, we know that $\gamma|\D|^{\alpha}$ is regular and self-adjoint, and $T (1+(\gamma|\D|^{\alpha})^2)^{-\frac12}$ is compact for any $T\in\Lip^0(\D)$. 
Moreover, since  $\Lip(\gamma|\D|^{\alpha}) = \Lip(|\D|^{\alpha})$ contains $\Lip(\D)$, 
we see that $\pi(A) \subset \overline{\Lip^0(\D)}\subset \overline{\Lip^0(\gamma|\D|^{\alpha})}$. 
Thus $(E,\gamma|\D|^{\alpha})$ is indeed an unbounded cycle. 
We note that $\gamma$ provides a spectral symmetry for $(E,\gamma|\D|^{\alpha})$. 
The operator-homotopy is obtained by composing the operator-homotopy between $\D$ and $\sgnmod^{\alpha}(\D)$ (see \cref{prop:power-damp_Lip_reg,remark:sgnmod}) with the operator-homotopy given for $t\in[0,1]$ by 
\begin{align}
\label{eq:op-hom_spec_symm}
\D_t := \cos\big(\tfrac{\pi t}2\big) \sgnmod^{\alpha}( \D) + \sin\big(\tfrac{\pi t}2\big) \gamma |\D|^{\alpha} . 
\end{align}
Note that $\gamma$ anti-commutes with $\sgnmod^{\alpha}(\D)$ as the latter is given by an odd function of $\D$ (see \cref{lem:odd_funct_calc}).
We then compute that $\D_t^2 = |\D|^{2\alpha}$, and thus $\Lip^0(\D) \subset \Lip^0(\D_t)$ for all $t\in[0,1]$, so $\D_t$ is indeed an operator-homotopy.
\end{proof}

As in \cite[Definition 3.1]{DGM18}, we say that an unbounded cycle $(E,\D)$ is \emph{weakly degenerate} if $\D$ is given by a sum $\D = \D_0+\mS$, such that 
\begin{itemize}
\item $\D_0$ and $\mS$ are odd regular self-adjoint operators with $\Dom\D = \Dom\D_0\cap\Dom \mS$;
\item $\mS$ is invertible, $A\subset\Lip(\mS)$, and $\mS a-a\mS=0$ for all $a\in A$;
\item there is a common core $\E \subset \Dom(\mS\D_0)\cap\Dom(\D_0 \mS)$  for $\D_0$ and $\mS$ such that $\D_0\mS+\mS\D_0=0$ on $\E$. 
\end{itemize}
Roughly speaking, this means that $\mS$ is degenerate and $\D_0$ has Clifford symmetry $\gamma=\sgn(\mS)$. The proof of \cref{lem:Cliff_symm_spec_symm} can be adapted to weakly degenerate cycles. 

\begin{lem}
\label{lem:weakly_deg}
Any weakly degenerate unbounded $A$-$B$-cycle $(E,\D=\D_0+\mS)$ is operator-homotopic to the spectrally symmetric unbounded  $A$-$B$-cycle $(E,\sgn(\mS) |\D|^\alpha)$ for any $0<\alpha<1$. 
In particular, $(E,\D)$ is null-homotopic. 
\end{lem}
\begin{proof}
The proof is the same as for \cref{lem:Cliff_symm_spec_symm}, but we need to show that \cref{eq:op-hom_spec_symm} is again an operator-homotopy (with $\gamma = \sgn(\mS)$). 
We compute 
\[
\D_t^2 = |\D|^{2\alpha} + 2 \sin\big(\tfrac{\pi t}2\big) \cos\big(\tfrac{\pi t}2\big) [\sgnmod^\alpha(\D),\gamma] |\D|^\alpha .
\]
Since $\mS$ is invertible, also $\D$ is invertible, and we find that 
\[
[\sgnmod^\alpha(\D),\gamma] = [\D,\gamma] |\D|^{\alpha-1} = 2 |\mS| \, |\D|^{\alpha-1} . 
\]
In particular, $[\sgnmod^\alpha(\D),\gamma] |\D|^\alpha$ is a positive operator and therefore $\D_t^2 \geq |\D|^{2\alpha}$ for all $t\in[0,1]$. Hence, if $T(1+\D^2)^{-\frac12}$ is compact for some $T\in\End_B(E)$, then also $T(1+|\D|^{2\alpha})^{-\frac12}$ is compact, and therefore 
\[
T (1+\D_t^2)^{-\frac12} = T (1+|\D|^{2\alpha})^{-\frac12} (1+|\D|^{2\alpha})^{\frac12} (1+\D_t^2)^{-\frac12}
\]
is compact. Thus $\Lip^0(\D) \subset \Lip^0(\D_t)$ for all $t\in[0,1]$, so $\D_t$ is indeed an operator-homotopy.
Finally, it follows from \cref{coro:spec_symm_deg_0} that $(E,\gamma |\D|^\alpha)$ is null-homotopic. 
\end{proof}

\subsection{The unbounded \texorpdfstring{$\KK$}{KK}-group}
\label{sec:group}

As mentioned in \cref{remark:KK_group}, the isomorphism $\overline{\UKK}(A,B) \simeq \KK(A,B)$ from \cref{thm:bdd_transform_KK_isom} implies in particular that $\overline{\UKK}(A,B)$ is a group. Here we give a direct proof of this fact, working only in the unbounded picture of $\KK$-theory (hence avoiding the bounded transform entirely). In particular, the proof we give here (in contrast with \cref{thm:bdd_transform_KK_isom}) does not require the assumption that $A$ is separable. 

Given an unbounded $A$-$B$-cycle $(\pi,E,\D)$, define its `inverse' as 
\[
-(\pi,E,\D) := (\pi^\op,E^\op,-\D) , 
\]
where $E^\op = E$ with the opposite grading and the representation $\pi^\op(a) = (-1)^{\deg a} \pi(a)$ for homogeneous elements $a\in A$. 

\begin{thm}
\label{thm:UKK_group}
For any $\sigma$-unital $C^*$-algebras $A$ and $B$, the abelian semigroup $\overline{\UKK}(A,B)$ is in fact a group. 
More precisely, the inverse of $[(\pi,E,\D)] \in \overline{\UKK}(A,B)$ is given by $[-(\pi,E,\D)]$. 
\end{thm}
\begin{proof}
The sum $(\pi,E,\D) - (\pi,E,\D)$ is given by the Clifford symmetric cycle 
\begin{align*}
(\pi,E,\D) - (\pi,E,\D) &= \left( \pi\oplus\pi^\op , E\oplus E^\op , \mattwo{\D}{0}{0}{-\D} \right) , & 
\gamma &= \mattwo{0}{1}{1}{0} ,
\end{align*}
where $\gamma$ denotes the Clifford symmetry. From \cref{lem:Cliff_symm_spec_symm} we know that a Clifford symmetric cycle is operator-homotopic to a spectrally symmetric cycle. Furthermore, by \cref{coro:spec_symm_deg_0}, every spectrally symmetric cycle is null-homotopic. 
Thus we have shown that $(\pi,E,\D) - (\pi,E,\D)$ is null-homotopic, and therefore $[-(\pi,E,\D)]$ is indeed the inverse of $[(\pi,E,\D)]$. 
\end{proof}

\appendix 
\section{Appendix: On localisations of dense submodules}
\label{sec:appendix}

Let $X$ be a locally compact Hausdorff space, $B$ a $C^*$-algebra, and $E$ a Hilbert $C_0(X,B)$-module. 
We will show in this Appendix that a submodule of $E$ is dense if and only if it is pointwise dense. One way to prove this could be by showing that $E$ can be viewed as a continuous field of Banach spaces (where each Banach space is in fact a Hilbert $B$-module), and then applying the theory of continuous fields \cite{DD63} (for this approach, see for instance \cite[2.7, 2.8, \& 2.21]{Ebe16pre}). 
Here, we prefer instead to give our proof in the language of Hilbert $C^*$-modules. 

For $x\in X$ we denote by $\ev_{x} \colon C_0(X,B)\to B$ the $*$-homomorphism $f\mapsto f(x)$. 
Let $\iota\colon B\to B^{+}$ be the embedding of $B$ into its (minimal) unitisation $B^{+}$. 
We define the \emph{localisation} $E_{x}:=E\hot_{\ev_{x}} B^{+}$, and we note that there is a map $E\to E_{x}$ via $e\mapsto e_{x}:=e\hot 1$. 
For a submodule $F\subset E$ we write 
$$F_{x}:=\{f_{x}\in E_{x} \mvert f\in F\}\subset E_{x},$$ for the image of $F$ under the map $e\mapsto e_{x}$.
We collect some basic facts regarding these localisations in the following lemma. 

\begin{lem}
\label{lem:localisations}
\begin{enumerate}
\item The Hilbert $C_0(X,B)$-module $E$ is a central bimodule over $C_0(X)$, and the left $C_0(X)$ action is by adjointable operators.
\item The map $E\to E_{x}$ given by $e\mapsto e_{x}:=e\hot 1$ is surjective. 
\item We have a unitary isomorphism $E_x \simeq E\hot_{\ev_{x}}B$. 
\item We have the equality $\|e\|_{E}=\sup_{x\in X} \|e_{x}\|$, and the map $x\mapsto \|e_{x}\|$ lies in $C_0(X)$.
\end{enumerate}
\end{lem}
\begin{proof}
For (1), see for instance \cite[Definition 1.5]{Kas88} and the discussion following it. 
For (2), it suffices to consider elements $e\hot b\in E_x$ with $e\in E$ and $b\in B$. Picking $f\in C_0(X)$ such that $f(x)=1$ and defining $\til b\in C_0(X,B)$ by $\til b(x) := f(x) b$, we see that $e\hot b = e\til b \hot 1$, which proves (2). 
For (3), we note that the map $\textnormal{id}\hot\iota\colon E\hot_{\ev_{x}}B\to E\hot_{\ev_{x}}B^{+}$ is an isometry, so we only need to check that the range is dense. Using an approximate unit $u_{n}\in B$, we indeed find 
\[\|e\hot 1 -e\hot u_{n}\|^{2}=\|e\hot (1-u_n)\|^{2}=\|(1-u_{n})\ev_x(\langle e,e\rangle)(1-u_n)\|\to 0 . \]
The equality in (4) follows by direct calculation: 
\[\|e\|_{E}^{2}=\|\langle e,e \rangle\|_{C_0(X,B)}=\sup_{x\in X}\|\langle e,e\rangle (x)\|_{B}=\sup_{x\in X}\|\langle e\hot 1, e\hot 1\rangle_{E_{x}}\|_{B}=\sup_{x\in X} \|e_{x}\|^2.\]
Finally, for continuity of the norm, we use that 
$\|e_{x}\|=\|\langle e,e\rangle^{\frac{1}{2}}(x)\|$ and that the map $x\mapsto \langle e,e\rangle ^{\frac{1}{2}}(x)$ is continuous.
\end{proof}

\begin{prop}
\label{prop:dense_submodules}
If $F \subset E$ is a submodule, then $F$ is dense in $E$ if and only if for each $x\in X$, $F_{x}$ is dense in $E_{x}$. 
\end{prop}
\begin{proof}
We will freely use the facts from \cref{lem:localisations}. 
If $F$ is dense in $E$, the equality $\|e\|_{E}=\sup_{x\in X} \|e_{x}\|$ shows that $F_x$ is dense in $E_x$ for each $x\in X$. 
Conversely, suppose $F_x$ is dense in $E_x$ for all $x\in X$. 
Fix $\varepsilon>0$ and $\psi\in E$. 
For each $x\in X$, there exists $\phi\in F$ such that $\|\psi_x-\phi_x\| < \frac\varepsilon2$. By continuity of the norm, there exists a precompact open neighbourhood $U_x$ of $x$ in $X$ such that $$\sup_{y\in U_x} \|\psi_y-\phi_y\| < \varepsilon.$$ 
There exists a compact subset $K\subset X$ such that $\sup_{x\in X\backslash K} \|\psi(x)\| < \varepsilon$. 
By compactness of $K$, we can choose finitely many points $\{x_i\}_{i=1}^N$ such that $K \subset \bigcup_{i=1}^N U_{x_i}$. 
Thus on each $U_i := U_{x_i}$ there exists $\phi_{i} \in F$ such that $\sup_{y\in U_i} \|\psi_y-\phi_{i,y}\| < \varepsilon$. 
Let $U_0 := X\backslash K$, and let $\chi_{i}$ be a partition of unity subordinate to $\{U_i\}_{i=0}^N$. 
Let $\{u_n\}$ be an approximate unit for $B$, and choose $n$ large enough such that $\| \phi_{i,y} - \phi_{i,y} u_n \| < \varepsilon$ for all $i=1,\ldots,N$ and $y\in U_i$. 
Let $\eta_i \in C_0(X,B)$ be given by $\eta_i(x) := \chi_i(x) u_n$. Then the element $\phi := \sum_{i=1}^N \phi_i \eta_i \in F$ is supported on $V := \bigcup_{i=1}^N U_i$, and we compute 
\begin{align*}
\|\psi-\phi\| &\leq \sup_{x\in V}\|\psi_x-\phi_x\| + \sup_{x\in X\backslash V}\|\psi_x-\phi_x\| \\
&\leq \sup_{x\in V\backslash K}\Big\| \Big( 1 - \sum_{i=1}^N \chi_i(x) \Big) \psi_x \Big\| + \sup_{x\in V}\Big\| \sum_{i=1}^N \chi_i(x) (\psi_x-\phi_{i,x}) \Big\| \\
&\qquad+ \sup_{x\in V}\Big\| \sum_{i=1}^N \chi_i(x) (\phi_{i,x} - \phi_{i,x} u_n) \Big\| + \sup_{x\in X\backslash V}\|\psi_x\| \\
&\leq 4\varepsilon .
\end{align*}
It follows that $F$ is dense in $E$. 
\end{proof}

For any adjointable operator $T$ on $E$, we write $T_x := \ev_x(T) := T\hot1$ for the corresponding operator on $E_x=E\hot_{\ev_{x}} B^{+}$. 

\begin{coro}
\label{locallydenserange}
Let $E$ be a $C_0(X,B)$-module and $h\in\End_{B}^*(E)$. Then $h$ has dense range in $E$ if and only if for all $x\in X$, $h_{x}$ has dense range in $E_{x}$.
\end{coro}


\providecommand{\noopsort}[1]{}\providecommand{\vannoopsort}[1]{}
\providecommand{\bysame}{\leavevmode\hbox to3em{\hrulefill}\thinspace}
\providecommand{\MR}{\relax\ifhmode\unskip\space\fi MR }
\providecommand{\MRhref}[2]{%
  \href{http://www.ams.org/mathscinet-getitem?mr=#1}{#2}
}
\providecommand{\href}[2]{#2}

\end{document}